\documentclass{amsart}
\usepackage[fontsize=11pt]{fontsize}

\usepackage{amsthm, amsmath, amssymb, mathrsfs, enumitem, svn, amsfonts, stmaryrd, mathtools}
\usepackage{verbatim, comment} 

\usepackage{mathscinet}
\usepackage{tikz-cd} 
 
 \usepackage[a4paper, left=1.5cm, right=1.5cm, top=0.6cm,bottom=0.7cm]{geometry} 

\usepackage{tabu}  
\usepackage[all, cmtip]{xy}

\usepackage{color}
\usepackage[linktocpage=true, hidelinks, colorlinks, linkcolor=blue, citecolor=magenta, bookmarksopen=true]{hyperref} 

\newtheorem{theorem}[subsubsection]{Theorem}
\newtheorem{thm}[subsubsection]{Theorem}
\newtheorem{lemma}[subsubsection]{Lemma}

\newtheorem{cor}[subsubsection]{Corollary}

\newtheorem{prop}[subsubsection]{Proposition}

\theoremstyle{definition}
\newtheorem{defn}[subsubsection]{Definition}
\newtheorem{construction}[subsubsection]{Construction}

\newtheorem{example}[subsubsection]{Example}

\newtheorem{notation}[subsubsection]{Notation}

\newtheorem{convention}[subsubsection]{Convention}
 
\newtheorem{remark}[subsubsection]{Remark}
\newtheorem{rem}[subsubsection]{Remark}
\numberwithin{equation}{subsection}
\def \into {\hookrightarrow }
\def \to {\rightarrow}
\def \onto {\twoheadrightarrow}
\renewcommand{\projlim}{\varprojlim}

\newcommand{\GL}{\mathrm{GL}}
\newcommand{\Aut}{\mathrm{Aut}}

\def\Mat{\mathrm{Mat}}

\newcommand{\nc}{{\mathrm{nc}}}

\DeclareMathOperator{\rep}{Rep}

\DeclareMathOperator{\Ker}{Ker}

\DeclareMathOperator{\Gal}{Gal}
\DeclareMathOperator{\gal}{Gal}

\DeclareMathOperator{\Fil}{Fil}
\DeclareMathOperator{\fil}{Fil}
\DeclareMathOperator{\gr}{gr}

\def\et{\mathrm{\acute{e}t}}

\def\geom{{\mathrm{geom}}}

\newcommand{\Lie}{\mathrm{Lie}}

\newcommand{\kinfty}{{K_{\infty}}}
\newcommand{\hatkinfty}{{\widehat{K_{\infty}}}}
\newcommand{\gammak}{{\Gamma_K}}
\newcommand{\gammabbk}{{\Gamma_{\mathbb{K}}}}

\newcommand{\dan}{\text{$\mbox{-}\mathrm{an}$}}

\renewcommand{\log}{\mathrm{log}}

\def\dR{{\mathrm{dR}}}
\def\HT{{\mathrm{HT}}}
\def\Sen{{\mathrm{Sen}}}

\def\dif{{\mathrm{dif}}}

\newcommand{\Ainf}{{\mathbf{A}_{\mathrm{inf}}}}

\newcommand{\bdrplus}{{\mathbf{B}^+_{\mathrm{dR}}}}

\newcommand{\bdr}{{\mathbf{B}_{\mathrm{dR}}}}
\newcommand{\bdrnabla}{\mathbf{B}^{\nabla}_{\mathrm{dR}}}

\newcommand{\bht}{\mathbf{B}_{\mathrm{HT}}}

\newcommand{\obdr}{\mathcal{O}\mathbf{B}_{\mathrm{dR}}}
\newcommand{\obdrplus}{\mathcal{O}\mathbf{B}^+_{\mathrm{dR}}}
\newcommand{\obht}{\mathcal{O}\mathbf{B}_{\mathrm{HT}}}

\newcommand{\bbdr}{\mathbb{B}_{\mathrm{dR}}}
\newcommand{\bbdrplus}{\mathbb{B}_{\mathrm{dR}}^{+}}

\newcommand{\bbht}{\mathbb{B}_{\mathrm{HT}}}

 \def \O {{\mathcal{O}}}

\newcommand{\calH}{\mathcal{H}}

   \def \calO {{\mathcal{O}}}
     \def \coc {{\mathcal{O}C}}

\def \ok {{\mathcal{O}_K}}

\def \oc {{\mathcal{O}_C}}

\def \oko {{\mathcal{O}_{K_0}}}

\newcommand*{\wt}[1]{\widetilde{#1}}
\newcommand*{\wh}[1]{\widehat{#1}}

\newcommand{\Zp}{{\mathbb{Z}_p}}
\newcommand{\Qp}{{\mathbb{Q}_p}}
\newcommand{\Fp}{{\mathbb{F}_p}}
\newcommand{\zp}{{\mathbb{Z}_p}}
\newcommand{\qp}{{\mathbb{Q}_p}}
\newcommand{\fp}{{\mathbb{F}_p}}

\newcommand{\ch}{{\mathcal H}}

\newcommand{\cm}{\mathcal{M}}

 \newcommand{\gk}{{G_K}}

\newcommand{\cbf}{\mathbf{c}}

\newcommand{\kbf}{\mathbf{k}}

 \def \kpf {{K^{\mathrm{pf}}}}
\newcommand{\pf}{{\mathrm{pf}}}
 \newcommand{\bbk}{{\mathbb{K}}}
  \newcommand{\bbc}{{\mathbb{C}}}
    \newcommand{\bbC}{{\mathbb{C}}}

 \newcommand{\bbK}{{\mathbb{K}}}

\newcommand{\bbkinfty}{{\bbk_\infty}}
\newcommand{\hatbbkinfty}{{\widehat{\bbk_\infty}}}

 \newcommand{\bbl}{{\mathbb{L}}}

\renewcommand{\phi}{\varphi}

\newcommand{\ocflat}{{\mathcal{O}_C^\flat}}

\newcommand{\crh}{\mathcal{RH}}
\newcommand{\rrh}{\mathrm{RH}}

\newcommand{\bbldrplus}{\mathbb{L}_\dR^+}
\newcommand{\bbldr}{\mathbb{L}_\dR}

\newcommand{\oldr}{\mathcal{O}\mathbf{L}_\dR}

\newcommand{\Ht}{{H_{\underline{t}}}}
\newcommand{\hatlinfty}{{\widehat{L_\infty}}}

\begin{document}

\title[]{On $p$-adic Simpson and Riemann--Hilbert correspondences  in the imperfect residue field case}

\date{\today}


  \author[]{Hui Gao}   \address{Department of Mathematics and Shenzhen International Center for Mathematics, Southern University of Science and Technology, Shenzhen 518055, China}   \email{gaoh@sustech.edu.cn}


\begin{abstract}
Let $K$ be a mixed characteristic complete discrete valuation field with residue field admitting a finite $p$-basis, and let $G_K$ be the Galois group.  
Inspired by Liu and Zhu's construction of  $p$-adic Simpson and Riemann--Hilbert  correspondences  over rigid analytic varieties, we construct such correspondences for representations of $G_K$.
As an application, we prove a  Hodge--Tate (resp. de Rham) ``rigidity" theorem for $p$-adic representations of $G_K$, generalizing a result of Morita.
\end{abstract} 

\maketitle
\tableofcontents

\section{Introduction}

\subsection{Overview and main theorems}

Let $K$ be a CDVF (complete discrete valuation field) of characteristic zero  with residue field $k_K$ of characteristic $p$ such that $[k_K:k_K^p]< \infty$, and let $G_K$ be its absolute Galois group. 
The scope of this paper is to study $p$-adic representations of $G_K$.
Such representations  naturally arise in the study of \emph{relative $p$-adic Hodge theory}.
For example, let $R=\Zp\langle x^{\pm 1}\rangle$ be the (integral) ring corresponding to a rigid torus, and let $\wh{R_{(p)}}$ be the $p$-adic completion of the localization $R_{(p)}$. Then $\wh{R_{(p)}}[1/p]$ is a CDVF with residue field $\Fp(x)$.

 


It turns out that to understand  $p$-adic representations of $\pi_1^{\et}(R[1/p])$  (called \emph{``the relative case"} in the following), it is crucial to understand representations  for these CDVFs   with  imperfect  residue fields  (called \emph{``the imperfect residue field case"} in the following).
Such examples are indeed numerous. For example, Brinon studies the Hodge--Tate, de Rham, and crystalline representations in the imperfect residue field case  in \cite{Bri06}, which then pave the way for the studies in the relative case in \cite{Bri08}. 
For another example, very recently, Shimizu proves a variant of $p$-adic local monodromy theorem in the relative case in \cite{Shipst}; the proof makes crucial use of $p$-adic local monodromy theorem  in the imperfect residue field case  established in \cite{Moritacrys, Ohk13}.

Contrary to the above examples where results in the imperfect residue field case  are  obtained first and then are used in the study of relative case, this paper goes in the opposite direction.
Indeed,  inspired by the work of Liu-Zhu \cite{LZ17} where they construct a $p$-adic Simpson correspondence and a  $p$-adic Riemann--Hilbert correspondence in the relative case, we construct in this paper such correspondences in the imperfect residue field case. 
Before we discuss such correspondences, we start with an application  which is indeed a main motivation of this paper.  
This application, as will be discussed in the ensuing remark, in turn is motivated by considerations in \emph{integral $p$-adic Hodge theory}.
We start by introducing some notations.

\begin{notation}\label{notakl}
Fix an algebraic closure $\overline{K}$ of $K$ and let $G_K:=\gal(\overline{K}/K)$.  Throughout this paper, we use $L$ to denote another mixed characteristic CDVF whose residue field has a finite $p$-basis, and similarly fix some  algebraic closure  $\overline L$ and denote $G_L:=\gal(\overline{L}/L)$. Furthermore, we assume there is some non-zero (hence injective) \emph{continuous} field homomorphism  
 $$i: K \into L.$$  
Fix a such $i$, and extend it to some embedding (still denoted by $i$) of their algebraic closures 
$$i: \overline K \into \overline L.$$
(Such extensions are not unique, but we fix one.)
 With respect to these fixed field embeddings, we obtain a continuous group homomorphism (still denoted by $i$) of their Galois groups 
$$i: G_L \to G_K.$$ 
Given a $p$-adic Galois representation $V$ of $G_K$, we use $V|_{G_L}$ to denote the $p$-adic Galois representation of $G_L$  via the map $i$, and call it the ``restriction" of $V$ to $G_L$ (note though the map  $i: G_L \to G_K$ is not necessarily injective).
\end{notation}

 Given a $p$-adic Galois representation $V$ of $G_K$ of dimension $n$ as above, Brinon defines the Fontaine modules $D_{\HT, K}(V)$ resp. $D_{\dR, K}(V)$ (cf. \S \ref{subsecFonMod}): they are both   $K$-vector spaces of dimension $\leq n$, and $V$ is called Hodge--Tate resp. de Rham if   $\dim_K D_{\HT, K}(V) =n$ resp. $\dim_K D_{\dR, K}(V) =n$. One can similarly define $D_{\HT, L}, D_{\dR, L}$.

\begin{theorem}\label{thmintrorigi} (cf. \S \ref{subsecHT} and \S \ref{subsecdR}.)
Use the above notations. 
\begin{enumerate}
\item We have ``base-change" isomorphisms
\begin{align*}
D_{\HT, K}(V) \otimes_K L & \simeq D_{\HT, L}(V|_{G_L}), \\
  D_{\dR, K}(V) \otimes_K L & \simeq D_{\dR, L}(V|_{G_L}).
\end{align*} 

\item  The $G_K$-representation $V$ is Hodge--Tate (resp. de Rham) if and only if $V|_{G_L}$ (considered as a representation of $G_L$) is so.
\end{enumerate}
\end{theorem}

\begin{remark} \label{remhistory}
We   comment on the history, motivation, and  future outlook of the above theorem.
\begin{enumerate}
\item Note that Item (2) of Thm. \ref{thmintrorigi} is a trivial consequence of Item (1).
We call Item (2) the Hodge--Tate \emph{rigidity} theorem (resp. de Rham \emph{rigidity} theorem) for $p$-adic Galois representations in the imperfect residue field case. The terminology ``rigidity" (sometimes also called ``purity") comes from similar results in the relative case established by Liu-Zhu, Shimizu and Petrov; cf. Rem. \ref{remrigirelative} for more details.


\item Item (2) of Thm. \ref{thmintrorigi} was proved by Morita \cite{MoritadR} for certain special type embeddings denoted as $K \into \kpf$ (cf. Example \ref{notaspecialbbk}). 
Indeed, in Morita's setting,  $\kpf$ is the ``perfection" field of $K$ (hence in particular $\kpf$ has perfect residue field); furthermore, this embedding has many advantageous features (cf. Rem. \ref{remAdvantagemorita})  that allow Morita to obtain the rigidity theorems by an \emph{explicit} method. However, even in Morita's set-up, he did not obtain Item (1) of Thm. \ref{thmintrorigi}; furthermore, as far as we understand, Morita's method can not work in the general setting.


\item \label{item3} There are natural reasons to consider all possible embeddings $i: K \into L$. Indeed, in the  study  of crystalline (resp. semi-stable) representations, particularly in \emph{integral} $p$-adic Hodge theory, a  ``Frobenius-compatible"  embedding (introduced by Brinon) that we denote as $i_\sigma: K\into \bbk_\sigma$ (cf. Example \ref{examBT} for details) naturally arises, and is not covered by Morita's setting.

\item For \emph{integral} $p$-adic Hodge theory, we have better understanding in the imperfect residue field case (cf. \cite{BT08, Gaoimperf}) than in the  relative case: a major reason being Kedlaya's slope filtration theorem \cite{Ked04, Ked05} is only available in the former setting. As  Thm. \ref{thmintrorigi} is applicable in the natural setting of {integral}  theory in the {imperfect residue field case} (as mentioned in Item \eqref{item3}), we speculate it would be of further use in the development of   relative integral  theory.
\end{enumerate}
\end{remark}



In contrast to Morita's explicit method, our approach is completely conceptual, and is inspired by the work of Liu-Zhu \cite{LZ17} (and subsequent works such as \cite{DLLZ, Shi18, Pet23}).
Let us first quickly recall the $p$-adic Simpson and Riemann--Hilbert correspondences of Liu--Zhu.
Let $k/\qp$ be a finite extension, let $k_\infty$ be the extension by adjoining all $p$-power roots of unity, and let $\widehat{k_\infty}$ be the $p$-adic completion. With respect to the perfectoid field  $\widehat{k_\infty}$, one can construct the usual Fontaine ring $\bdr(\widehat{k_\infty})$. (Instead of $k_\infty$ and $\widehat{k_\infty}$, one can also use the algebraic closure $\bar k$ and its completion $\hat{\bar k}$).
Since we only use Thm. \ref{thmintroLZ17} and Rem. \ref{remrigirelative} for motivational reasons, we refer the readers to the references cited therein for more details.

\begin{theorem}\label{thmintroLZ17}
 \cite{LZ17}
Let $X$ be a smooth rigid analytic variety over $k$.  Then 
\begin{enumerate}
\item there is a tensor functor $\calH$   from the category of $\qp$-local systems on $X_{\et}$ to the category of nilpotent Higgs bundles on $X_{\widehat{k_\infty}}$.
In addition,

\item there is a tensor functor $\mathcal{RH}$ from the category of $\qp$-local systems on $X_{\et}$ to the category of filtered vector bundles on the ringed space
$\mathcal X = (X_{\widehat{k_\infty}},\calO_X\hat\otimes\bdr(\widehat{k_\infty}))$, equipped with a semi-linear action of $\Gal( {k_\infty}/k)$, and   an integrable connection that  satisfy Griffiths transversality.
\end{enumerate}
The functors $\calH$ and $\crh$ are  called the $p$-adic Simpson correspondence and the $p$-adic Riemann--Hilbert correspondence respectively. They are both compatible with pullback along arbitrary morphisms and  pushforward under smooth proper morphisms.
\end{theorem}

\begin{remark} \label{remrigirelative}
Let us summarize some ``rigidity" theorems in the relative case. Let $X$ be as in above theorem and suppose it is \emph{geometrically connected}, let $\bbl$ be an \'etale local system.
\begin{enumerate}
\item  Liu--Zhu \cite[Thm. 1.3]{LZ17} shows that $\bbl$ is de Rham if and only if its specialization to \emph{one} classical point  is so. The key argument is that they can prove $D_\dR(\bbl)$ is a \emph{vector bundle} and satisfies a \emph{base change property}.

\item Shimizu \cite[Thm. 1.1]{Shi18} shows that  the generalized Hodge--Tate weights of the $p$-adic Galois representation  $\bbl_{\bar x}$ (the specialization of $\bbl$ at any classical point $x$)  are  {constant}.
A key ingredient in the proof is Shimizu's construction of \emph{decompleted} Simpson   and Riemann--Hilbert correspondences (denote by $\mathrm{H}$ and  $\rrh$ respectively).
Another interesting result is \cite[Thm. 1.5]{Shi18}, which shows that if $\bbl$ is of rank $\leq 2$, then it is Hodge--Tate if and only if  its specialization to \emph{one} classical point is so. The restriction on the rank seems difficult to remove; however, see next item.

\item If in addition  $X$ comes from analytification of a smooth  \emph{algebraic} variety over $k$, then Petrov \cite[Prop. 1.2]{Pet23} shows that $\bbl$ is Hodge--Tate if and only if  its specialization to \emph{one} classical point is so. Petrov's proof makes crucial use of the log-Simpson and log-Riemann--Hilbert correspondences constructed in \cite{DLLZ}.
\end{enumerate}
\end{remark}

Now, let us state our $p$-adic Simpson and Riemann--Hilbert correspondences in the imperfect residue field case.
 Let $K_\infty/K$ be the extension by adjoining all $p$-power roots of unity, and let $\hatkinfty$ be the $p$-adic completion. Let 
$$G_{\kinfty}: =\gal(\overline{K}/\kinfty), \quad \Gamma_K :=\gal(\kinfty/K).$$ 
Let $C$ be the $p$-adic completion of $\overline{K}$, and let $\bdr=\bdr(C)$ be the ``usual" Fontaine's ring associated to the perfectoid field $C$, and let $t$ be the usual Fontaine's element.
Let $\obdr$ be the   de Rham period ring in the imperfect residue field case (which contains $\bdr$).
(There are some subtle issues about $\obdr$, cf. Rem. \ref{remcompleteobdr} for details). Let $\oldr:=(\obdr)^{G_\kinfty}$.
See Notation \ref{notarepsemilin} for  notations of the form ``$\rep_{\mathcal G}(R)$" (the usual category of semi-linear representations).



 \begin{theorem} \label{thmintromain}
 (cf. Thm. \ref{thmSimpson} and Thm. \ref{thmRH} for   precise statements).
There are two commutative diagrams of tensor functors where the vertical functors are tensor equivalences.
\begin{equation*}
\begin{tikzcd}
                        &                                                  & \rep_\gammak(\hatkinfty) \arrow[dd, "\simeq"] \\
\rep_\gk(\qp) \arrow[r] & \rep_\gk(C) \arrow[ru, "\calH"] \arrow[rd, "\mathrm H"'] &                                               \\
                        &                                                  & \rep_\gammak(\kinfty)                        
\end{tikzcd}
\end{equation*}
and
\begin{equation*}
\begin{tikzcd}
                        &                                                           & \rep_\gammak(\oldr) \arrow[dd, "\simeq"] \\
\rep_\gk(\qp) \arrow[r] & \rep_\gk(\bdr) \arrow[ru, "\crh"] \arrow[rd, "\rrh"'] &                                          \\
                        &                                                           & \rep_\gammak(\kinfty((t)))              
\end{tikzcd}
\end{equation*}
\begin{enumerate}
\item 
We call the functor $\calH$ (resp. $\mathrm{H}$) the Simpson correspondence (resp. the decompleted Simpson correspondence). Furthermore, we can upgrade $\calH$ so that it maps to a category of representations equipped with Higgs fields.

\item 
We call the functor $\crh$ (resp. $\rrh$) the Riemann--Hilbert correspondence (resp. the decompleted Riemann--Hilbert correspondence). Furthermore, we can upgrade $\crh$ so that it maps to a category of representations equipped with filtrations and   integrable connections satisfying Griffiths transversality.

\item Let $i: K \into L$ be the continuous embedding as in Notation \ref{notakl}.
All the functors $\calH, H, \crh, \rrh$ (as well as the upgraded versions) satisfy natural \emph{base change properties} with respect to $i$. For a quick example (the other cases are similar), let   $ W\in \rep_{G_K}(C)$, then we have
$$\calH(W) \otimes_{ \hatkinfty}  \hatlinfty  \simeq \calH_{L}(W|_{G_L}\otimes_C C_L)$$
where $ \hatlinfty, C_L $ (resp. $\calH_{L}$) are the analogous fields (resp. the analogous Simpson correspondence functor) for the field $L$. 

 \end{enumerate} 
 \end{theorem}

Once the above correspondences are established, it is relatively easy (compared to the results in Rem. \ref{remrigirelative}) to obtain Thm. \ref{thmintrorigi}. Indeed, the situation is  much simpler in our case: our $D_{\HT, K}(V)$ and $D_{\dR, K}(V)$ are automatically ``vector bundles"; indeed they are just finite dimensional vector spaces over $K$. It hence suffices to establish the base change property as in Thm. \ref{thmintrorigi}(1):  this will follow  from the  base change property of the Simpson and Riemann--Hilbert correspondences, together with some cohomological computations.

\begin{remark}We now comment on  the relation between Thm. \ref{thmintromain} and various ``Sen theories".
\begin{enumerate}
\item Let $\bbk$ be a CDVF with \emph{perfect} residue field, then  classical Sen theory says that we have a diagram of tensor functors:
\begin{equation*}
\begin{tikzcd}
                             &                                                                                        & \rep_\gammabbk(\hatbbkinfty) \arrow[dd, "\simeq"] \\
\rep_{G_\bbk}(\qp) \arrow[r] & \rep_{G_\bbk}(\bbC) \arrow[ru, "{\calH_\bbk, \simeq}"] \arrow[rd, "{H_\bbk, \simeq}"'] &                                                   \\
                             &                                                                                        & \rep_\gammabbk(\bbkinfty)                        
\end{tikzcd}
\end{equation*}
(Note  in contrast to the general case in Thm. \ref{thmintromain}, all the functors on the triangles are tensor \emph{equivalences}). 
Hence our $p$-adic Simpson   correspondence  could be regarded as a ``generalization" of the classical Sen theory. (Similarly, our Riemann--Hilbert correspondence could be regarded as a ``generalization" of the classical Sen--Fontaine theory reviewed in \S \ref{subsecSenFon}).

\item For general $K$, Brinon \cite{Bri03} also developed a ``Sen theory". Let us only very briefly recall here, cf. \S \ref{subsecBrinonSen} for more details.
Indeed, it is related to the embedding $K \into \kpf$ that we mentioned in Rem. \ref{remhistory}(2), where $\kpf$ is the ``perfection" field of $K$. 
Let  $K_\infty^\pf$ is the cyclotomic extension with $\widehat{K_\infty^\pf}$ being the $p$-adic completion, and let $\Gamma:=\Aut(C/K_\infty^\pf)$. Then Brinon constructs a tensor equivalence (there is also a ``decompleted" version) 
$$\rep_{G_K}(C) \simeq \rep_\Gamma(\widehat{K_\infty^\pf}).$$
Here are some general remarks:
\begin{enumerate}
\item   The embedding $K \into \kpf$, and hence  Brinon's theory, depend on choices of ``local coordinates" $t_i$, cf. Notation \ref{notatitflat}.
In addition, Brinon's theory has no \emph{functoriality} (or base change property) with respect to embeddings $i: K \into L$.

\item Morita's proof of Hodge--Tate rigidity  with respect to  $K \into \kpf$ (cf. Rem. \ref{remhistory}(2)) makes use of Brinon's theory.
\item The field $\widehat{\kinfty}$ in our Simpson correspondence is a  {subfield} of $\widehat{K_\infty^\pf}$, hence our theory can be regarded as a \emph{refinement} of Brinon's theory.
 Note in particular the field $\widehat{\kinfty}$ (unlike $\widehat{K_\infty^\pf}$) is independent of the ``local coordinates" $t_i$.
\end{enumerate}
 

\item Andreatta--Brinon \cite{AB10} then further developed a ``Sen--Fontaine theory" for general $K$, cf. \S \ref{subsecAB10}, which \emph{lifts} Brinon's Sen theory, and is used in Morita's proof of de Rham rigidity with respect to $K \into \kpf$.
In comparison, our Riemann--Hilbert correspondence   \emph{lifts}  the Simpson correspondence, which however is via a very \emph{different} lifting process. Indeed,  Andreatta--Brinon's theory is a ``$(\ker\theta)$-adic lifting", whereas ours is a ``$t$-adic lifting"; compare the diagrams in Rem. \ref{remRelationABBrinon} and Rem. \ref{remRelationHRH}.
\end{enumerate}
\end{remark}

Let us give some brief technical remarks about the proof of Thm. \ref{thmintromain}. The  equivalence $\rep_\gammak(\kinfty) \simeq \rep_\gammak(\hatkinfty)$ is proved using similar ideas as in the perfect residue field case, namely, via  Tate's normalized traces.
The key part then is to construct the functor $\rep_{G_K}(C) \to \rep_\gammak(\hatkinfty)$, and this is where we are inspired by Liu--Zhu's construction, which uses some   argument from Grothendieck's proof of $\ell$-adic monodromy theorem. Finally, the Riemann--Hilbert correspondence can be obtained by   d\'evissage argument.

\subsection{Structure of the paper}
In \S \ref{secfontaine}, we review (Hodge--Tate and de Rham) Fontaine rings and Fontaine modules; then in \S \ref{secKL}, we introduce similar notations with respect to the embedding $K \into L$. In \S \ref{secTateSen}, we discuss Sen theory and Tate--Sen decompletion, which are used in \S \ref{secSimpson} to construct the Simpson correspondence and to prove the Hodge--Tate rigidity theorem. Via d\'evissage argument, we discuss Sen--Fontaine theory and Tate--Sen--Fontaine decompletion in \S \ref{secSenFontaine}, which are used in \S \ref{secRH} to construct the Riemann--Hilbert correspondence and to prove the de Rham rigidity theorem.

\subsection{Notations and conventions}

\begin{notation}\label{notarepsemilin}
Suppose $\mathcal G$ is a topological group that acts continuously on a topological ring $R$. We use $\rep_{\mathcal G}(R)$ to denote the category where an object is a finite free $R$-module $M$ (topologized via the topology on $R$) with a continuous and \emph{semi-linear} $\mathcal G$-action in the usual sense that 
$$g(rx)=g(r)g(x), \forall g\in \mathcal G, r \in R, x\in M.$$
(The only case in this paper where the action is \emph{linear} is when $R=\qp$). Morphisms in the category are the obvious ones.
\end{notation}

\subsubsection{Locally analytic vectors}\label{subsubLAV}
Let us very quickly recall the notion of {locally analytic vectors}, which is  particularly convenient  in the study of \emph{Sen theory}, see \cite[\S 2.1]{BC16} and \cite[\S 2]{Ber16} for more details. Recall the multi-index notations: if $\cbf = (c_1, \hdots,c_d)$ and $\kbf = (k_1,\hdots,k_d) \in \mathbb{N}^d$ (here $\mathbb{N}=\mathbb{Z}^{\geq 0}$), then we let $\cbf^\kbf = c_1^{k_1} \cdot \ldots \cdot c_d^{k_d}$.
Let $G$ be a $p$-adic Lie group, and let $(W, \|\cdot \|)$ be a $\Qp$-Banach representation of $G$.
Let $H$ be an open subgroup of $G$ such that there exist coordinates $c_1,\hdots,c_d : H \to \Zp$ giving rise to an analytic bijection $\cbf : H \to \Zp^d$.
 We say that an element $w \in W$ is an $H$-analytic vector if there exists a sequence $\{w_\kbf\}_{\kbf \in \mathbb{N}^d}$ with $w_\kbf \to 0$ in $W$, such that \[g(w) = \sum_{\kbf \in \mathbb{N}^d} \cbf(g)^\kbf w_\kbf, \quad \forall g \in H.\]
Let $W^{H\dan}$ denote the space of $H$-analytic vectors.
We say that a vector $w \in W$ is \emph{locally analytic} if there exists an open subgroup $H$ as above such that $w \in W^{H\dan}$.


\subsubsection{Some other conventions.} All group cohomologies are \emph{continuous} group cohomologies. We use $\Mat(\ast)$ to denote the set of matrices with elements in $\ast$; the size of the matrices will be clear from the context.

\subsection*{Acknowledgement}
We thank Laurent Berger, 
Olivier Brinon, Heng Du, Shizhang Li, Ruochuan Liu, Tong Liu, Shun Ohkubo, and Yimeng Tang
for some useful discussions and correspondences.
Special thanks to Yupeng Wang for several inspiring discussions and   comments.
We thank the anonymous referees for very careful readings and useful suggestions.
The author is partially supported by the National Natural Science Foundation of China under agreements   NSFC-12071201,  NSFC-12471011.

\section{Fontaine rings and Fontaine modules}\label{secfontaine} 

In this section, we review Fontaine rings and Fontaine modules in the imperfect residue field case.

\begin{notation} \label{nota: field K}
Let $K$ be a complete discrete valuation field of characteristic $0$ with residue field $k_K$ of characteristic $p>0$ such that $[k_K:k_K^p]=p^d$ where $d\geq 0$.  
Let $\ok$ be the ring of integers.
 \emph{Fix} an algebraic closure $\overline{K}$ of $K$, and let $G_{K}=\Gal(\overline{K}/K)$.
Let $C$ be the $p$-adic completion of $\overline K$.
Let $\varepsilon_0=1$, let $\varepsilon_1 \in \overline{K}$ be a primitive $p$-th root of unity, and inductively define a sequence $\{\varepsilon_n\}_{n\geq 0} $ in $ \overline{K}$ where $\varepsilon_{n+1}^p=\varepsilon_n, \forall n$.
\end{notation}

\begin{rem} We remark on the various \emph{choices} made in this paper.
\begin{enumerate}
\item One can readily check that different choices of $\overline{K}$ produce ``equivalent results"  for all constructions in this paper.

\item We will use the fixed $\{\varepsilon_n\}_{n\geq 0}$ to define Fontaine's usual element $t$. A different choice of these elements also produce ``equivalent results"  for all constructions in this paper.

\item In Notation \ref{notatitflat}, we will introduce a system of ``coordinates" $t_i$'s for $K$: these coordinates make it possible to write out ``explicit" expressions for the period rings such as in Prop. \ref{propucoor}. Note that 
\begin{enumerate}

\item The main result in \S \ref{secfontaine}, namely the Fontaine module theory Prop. \ref{propFonModule}, is \emph{independent}
 of these $t_i$'s. (So is the connection operator $\nabla$, cf. the paragraph above Notation \ref{drnabla}).

\item The embedding $K\into \kpf$ in Example \ref{notaspecialbbk}  not only \emph{depends} on $t_i$, it further depends on choices of $p$-power roots of them. This means that Brinon's Sen theory reviewed in \S \ref{subsecBrinonSen} and hence Morita's result Thm. \ref{thmmorita} \emph{depend} on these choices.

\item In contrast, our main results on Simpson and Riemann--Hilbert correspondences are  \emph{independent}
 of these coordinates $t_i$'s (although we do use them for intermediate computations).

\end{enumerate}

\end{enumerate}
\end{rem}

 \subsection{Fontaine rings}\label{sec21}

In this subsection, we review some Fontaine period rings in the imperfect residue case. 

\begin{convention}
In this paper, we choose to use the ``relative"-style notations (e.g., similar to those in \cite{LZ17}), since our main constructions are parallel to those in the relative case. Nonetheless, most of the results in this subsection are developed in \cite{Bri06}. Here, we compare the notations:
\begin{enumerate}
\item Our $C, \Ainf(\O_C/\O_{K})$ are the same as those in \cite{Bri06}. Our notation $\coc$ does not appear in \cite{Bri06} (but the ring  is used there).
\item Our $\bdr, \obdr^{+}, \obdr^{\nc}$ correspond to ``$\bdrnabla, \bdrplus, \bdr$" in \cite{Bri06} respectively. Our $ \obdr$ is not defined in \cite{Bri06}. (There is indeed some subtle issues about these de Rham period rings, cf. Rem. \ref{remcompleteobdr}.)

\item Our $\bht, \obht$ correspond to ``$\gr \bdrnabla, \bht$" in \cite{Bri06} respectively. 
\end{enumerate}
\end{convention}

\begin{notation}
 Since $\O_C$ is a perfectoid ring, we can define its tilt $\O_C^\flat$; let $W(\O_C^\flat)$ be the ring of Witt vectors. 
 Elements in $\O_C^\flat$ are in bijection with sequences $(x^{(n)})_{n\geq 0}$ where $x^{(n)} \in \O_C$ and $(x^{(n+1)})^p=x^{(n)}$. Let $\underline{\varepsilon} \in \O_C^\flat$  be the element defined by the sequence  $\{\varepsilon_n\}_{n\geq 0}$, and let $[\underline{\varepsilon}] \in W(\O_C^\flat)$ be its Teichm\"uller lift.
Let 
 $$\theta: W(\O_C^\flat) \to \oc$$
  be the usual Fontaine's map, it extends to
$$\theta: W(\O_C^\flat)[1/p] \to C.$$
Both these maps have principle kernels generated by $\frac{[\underline \varepsilon]-1}{[\underline \varepsilon^{1/p}]-1}$.
\end{notation}

\begin{notation}
Let $\bdrplus$ be the ($\ker \theta$)-adic completion of $W(\O_C^\flat)[1/p]$, and hence the $\theta$-map extends to
$$\theta: \bdrplus \to C.$$
Let $t:=\log [\underline \varepsilon] \in \bdrplus$ be the usual element, which is also a generator of the kernel of the $\theta$-map on $\bdrplus$.
 We  equip $\bdrplus$ with the $t$-adic filtration, and extend it to a filtration on $\bdr: =\bdrplus[1/t]$ via $\fil^i \bdr: =t^i \cdot \bdrplus, \forall i \in \mathbb Z$. Define 
$$\bht := \bigoplus_{i \in \mathbb Z} \gr^i(\bdr).  $$
Note that the $G_K$-action on $W(\O_C^\flat)$ naturally induces actions on $\bdr$ and $\bht$.
\end{notation}

\begin{defn}
 By scalar extension,  the $\theta$-map induces an ($\ok$-linear) map 
$$\theta_{K}:  \ok\otimes_{\mathbb{Z}} W(\ocflat) \to \oc.$$ Let 
\begin{equation} \label{eqainfocok}
\Ainf(\O_C/\O_{K}): = \projlim_{n \geq 0}  \left( \ok\otimes_{\mathbb{Z}} W(\O_C^\flat) \right)/\left(  \theta_{K}^{-1}(p\O_C)  \right)^n. 
\end{equation}
Then the $\theta_{K}$-map extends to 
$$\theta_{K}: \Ainf(\O_C/\O_{K})  \to \O_{C} \quad \text{ and } \quad \theta_{K}: \Ainf(\O_C/\O_{K}) \otimes_{\Zp} \Qp \to C.$$
 Let 
$$\obdrplus:= \projlim_{n \geq 1} (\Ainf(\O_C/\O_K) \otimes_{\Zp} \Qp)/\left(\Ker \theta_K \right)^n.  $$
Note that there is a canonical map $\bdrplus \to \obdrplus$, and hence we can regard $t$ as an element in $\obdrplus$.
The $\theta_K$-map extends to 
$$\theta_K: \obdrplus \to C.$$
Let $m_{\dR}:=\Ker \theta_K$ for the above map.
Define a $\mathbb{Z}^{\geq 0}$-filtration on $\obdrplus$ where 
 $$\mathrm{fil}^i \obdrplus:=m_{\dR}^i, \forall i \geq 0.$$
 Define
 $$\obdr^{ \nc}:=\obdrplus[1/t].$$
 (Here, the superscript ``nc" stands for ``non-complete", as will be explained in Rem. \ref{remcompleteobdr}.)
Using $\mathrm{fil}^i$ on $\obdrplus$, define a $\mathbb{Z}$-filtration on $\obdr^{\nc}$ by:
\begin{eqnarray*}
\Fil^0 \obdr^{\nc} &:=& \sum_{n=0}^\infty t^{-n}\cdot \mathrm{fil}^n \obdrplus \\
\Fil^i \obdr^{\nc}  &:=& t^i \cdot  \Fil^0\obdr^{\nc} , \forall i \in \mathbb{Z}.
\end{eqnarray*} 
$\obdr^{\nc} $ carries a natural $G_K$-action such that $\Fil^i \obdr^{\nc} $ is stable under the action.
\end{defn}

\begin{remark}\label{remcompleteobdr}
Given a ring $R$ with a  decreasing filtration $\{\fil^i R\}_{i \in \mathbb Z}$, we say this filtration is \emph{complete} if for each $i$,  we have $\projlim_{j \geq 1} \fil^i R/\fil^{i+j} R =\fil^i R$.
\begin{enumerate}
\item Note that the ring $\obdrplus$ is $t$-adically complete, hence the filtration   $\{t^i  \obdrplus\}_{i \in \mathbb Z}$ is a complete filtration on $\obdr^{ \nc}$. However, this is not the correct filtration for Fontaine module theory; the correct one is the $\Fil^i \obdr^{ \nc}$ filtration defined above. As we will see in below, the ring $\obdr^{\nc} $ together with this filtration is \emph{enough} to establish the usual Fontaine module theory, cf. Rem. \ref{rem223}.

\item 
However, a subtle issue is that the filtration $\Fil^i \obdr^{ \nc}$ is \emph{not} complete (unless $k_K$ is perfect)!  As observed by Shimizu and noted in \cite[Rem. 2.2.11]{DLLZ}, this causes trouble in the construction of Riemann--Hilbert correspondences, which is a \emph{lift} of the Simpson correspondence. Fortunately (cf. \emph{loc. cit.}), this issue can be resolved  by introducing the $t$-adic completions of the ring and its filtrations, cf. Def. \ref{notaobdr}.
\end{enumerate} 
\end{remark}

\begin{defn} \label{notaobdr}
Note that $\Fil^0 \obdr^{\nc}$ is a ring, and hence define its $t$-adic completion:
$$ \Fil^0 \obdr: =\projlim_{n >0}(\Fil^0 \obdr^{\nc})/t^n. $$
Let $$\obdr := (\Fil^0 \obdr)[1/t],$$
 and equip it with the $t$-adic filtration, namely
$$\fil^i \obdr := t^i \cdot \Fil^0 \obdr, \forall i \in \mathbb Z.$$
\end{defn}

\begin{remark}\label{remobdrplus}
Given a CDVF with \emph{perfect} residue field, then with notations in the classical setting (cf. Notation \ref{notabbkperf}), there is a \emph{co-incidence} $$\fil^0 \bbdr =\bbdrplus.$$ In the general imperfect residue field case (or in relative case), the natural maps
\[ \obdrplus \to  \fil^0 \obdr^\nc \to \fil^0 \obdr\]
are in general only injective. Indeed, we shall  use $\fil^0 \obdr$ as the ``effective" de Rham period ring.
\end{remark}

\begin{defn}
For $-\infty \leq a \leq b \leq +\infty$, define
\begin{eqnarray*}
 \obdr^{\nc, [a, b]} :=\fil^a \obdr^{\nc}/\fil^{b+1}\obdr^{\nc},\\
 \obdr^{  [a, b]} :=\fil^a \obdr /\fil^{b+1}\obdr,
\end{eqnarray*}
where $\fil^{-\infty} \obdr^{\nc} =\obdr^{\nc}, \fil^{+\infty} \obdr^{\nc}=0$ and similarly for $\obdr$. Note that since $\Fil^i \obdr$ is the $t$-adic completion of $\Fil^i \obdr^{\nc}$ for all $i \in \mathbb Z$, we have
$$\obdr^{\nc, [a, b]} = \obdr^{  [a, b]}, \forall a, b \neq \pm \infty.$$
In particular, we have $\gr^i \obdr^{\nc} \simeq \gr^i \obdr, \forall i \in \mathbb Z$. Thus, we can define 
$$\coc: = \gr^0 \obdr^{\nc} \simeq \gr^0 \obdr$$
$$\obht : =\bigoplus_{i \in \mathbb Z} \gr^i \obdr^{\nc} \simeq \bigoplus_{i \in \mathbb Z} \gr^i  \obdr$$
\end{defn}
 
\begin{remark}
The relations of all the ``de Rham rings" are summarized in the following diagram where all rows are short exact.
\begin{equation}
\begin{tikzcd}
0 \arrow[r] & \Fil^1 \bdrplus \arrow[r] \arrow[d, hook]                & \bdrplus \arrow[r, "\theta"] \arrow[d, hook]                                      & C \arrow[r] \arrow[d, "="]       & 0 \\
0 \arrow[r] & \mathrm{fil}^1 \obdrplus \arrow[r] \arrow[d, hook]                & \obdrplus \arrow[r, "\theta_K"] \arrow[d, hook]                                     & C \arrow[r] \arrow[d, hook]      & 0 \\
0 \arrow[r] & \Fil^1 \obdr^{\mathrm{nc}} \arrow[r] \arrow[d, hook] & \Fil^0 \obdr^{\mathrm{nc}} \arrow[r, "\bmod t "] \arrow[d, hook] & \calO C \arrow[r] \arrow[d, "="] & 0 \\
0 \arrow[r] & \Fil^1 \obdr \arrow[r]                               & \Fil^0 \obdr \arrow[r, "\bmod t "]                                             & \calO C \arrow[r]                & 0
\end{tikzcd}
\end{equation}
\begin{enumerate}
\item We claim in fact the three squares in the left side are all Cartesian: namely, the embeddings of the various ``de Rham rings" in the center column are in fact \emph{strict} with respect to their filtrations (where we use $\mathrm{fil}^i$-filtration  for $\obdrplus$). The only non-trivial part of the claim is about the square in the second floor, see \cite[Prop. 2.20]{Bri06} for the proof.

\item Note \cite[Prop. 2.20]{Bri06} cited above shows that the $m_\dR$-adic filtration on $\obdrplus$ \emph{coincides} with its \emph{induced} filtration from that on $\fil^0 \obdr$. Thus, in particular, $\obdrplus$ is a \emph{filtration complete subring}   inside $\fil^0 \obdr$. 

\item To ease the mind, let us mention that in our main construction  (in the Riemann--Hilbert correspondence), we only use the following small part of the above diagram: 
\begin{equation*}
\begin{tikzcd}
    0 \arrow[r]        &                       \Fil^1 \bdrplus \arrow[r]\arrow[d, hook] &                                     \bdrplus \arrow[r, "\bmod t"]\arrow[d, hook]  & C \arrow[d, hook] \arrow[r] & 0   \\
0 \arrow[r] & \Fil^1 \obdr \arrow[r] & \Fil^0 \obdr \arrow[r, "\bmod t"] & \calO C \arrow[r] & 0
\end{tikzcd}
\end{equation*}
\end{enumerate}
\end{remark}

\begin{notation}\label{notatitflat}
  Recall $[k_K: k_K^p]=p^d$. Let $\bar{t_1}, \cdots, \bar{t_d}$ be a $p$-basis of $k_K$, and for each $i$,  let $t_i \in \ok$ be some fixed lift of $\bar{t_i}$.
For each $m \geq 0$, fix an element $t_{i, m}\in \overline K$ where $t_{i, 0}=t_i$ and $(t_{i, m+1})^p =t_{i, m}, \forall m$. By \emph{abuse of notations}, we simply denote
$$t_i^{\frac{1}{p^m}}: = t_{i, m}.$$ 
Note that the sequence $\{t_i^{\frac{1}{p^m}}\}_{m \geq 0}$ defines an element in $\ocflat$, which we denote as $t_i^\flat$. Let $[t_i^\flat] \in W(\ocflat)$ be the Teichm\"uller lift. Let 
$$u_i =t_i \otimes 1 -1\otimes [t_i^\flat] \in \ok \otimes_{\mathbb{Z}} W(\O_C^\flat).$$
\end{notation}

We (briefly) recall some constructions above \cite[Prop. 2.9]{Bri06}.
There is a natural homomorphism $\bdrplus \to \obdrplus$. Note that $u_i \in m_{\dR}$ and $\obdrplus$ is complete with respect to the $m_{\dR}$-adic topology, hence there is a natural homomorphism
$$f: \bdrplus[[u_1, \cdots, u_d]] \to \obdrplus.$$

\begin{prop} \label{propucoor}
\cite[Prop. 2.9, Prop. 2.11]{Bri06} 
\begin{enumerate}
\item The above homomorphism $f: \bdrplus[[u_1, \cdots, u_d]] \to \obdrplus $ is an isomorphism; in addition, $f(u_i) \in \fil^1 \obdr^\nc$.

\item Let $\bar{u}_i$ be the image of $f(u_i)$ in  $\gr^1 \obdr^\nc$, then $f$ induces an isomorphism of graded rings
$$\obht \simeq C[t^{\pm 1}, \bar{u}_1, \cdots, \bar{u}_d]$$
where the right hand side is graded by total degree (where $t^{-1}$ has degree $-1$). Hence, there is an isomorphism 
$$\coc \simeq C[\frac{\bar{u}_1}{t}, \cdots, \frac{\bar{u}_d}{t}].$$
\end{enumerate}
\end{prop}

 \begin{notation}\label{notaVi}
Define
 $$V_i= \frac{1}{t} \log(\frac{[t_i^\flat]}{t_i}  ).$$
 These are elements of $\fil^0 \obdr^{\nc}$, and by abuse of notation we also regard $V_i$ as an element  in $\gr^0 \obdr^{\nc}=\coc$. Then by  Prop. \ref{propucoor}, we can deduce
 \begin{eqnarray*}
 \obdrplus  \simeq & \bdrplus[[tV_1, \cdots, tV_d]]  \\
 \coc  \simeq & C[V_1, \cdots, V_d]
 \end{eqnarray*}
 \end{notation}

\subsection{Fontaine modules} \label{subsecFonMod}
\begin{defn}
For $V \in \rep_\gk(\qp)$, define
\begin{eqnarray*}
D_{\dR, K}^\nc(V) = &(V\otimes_\qp \obdr^\nc)^{G_K}\\
\quad D_{\dR, K}(V) = &(V\otimes_\qp \obdr)^{G_K} \\
\quad D_{\HT, K}(V) = &(V\otimes_\qp \obht)^{G_K}.
\end{eqnarray*}
\end{defn}

\begin{prop}\label{propFonModule}
\begin{enumerate}
\item There are canonical isomorphisms 
$$(\obdr^{\nc})^{G_K}=(\obdr)^{G_K}=K, \quad (\obht)^{G_K}=K.
$$

\item The natural injection $D_\dR^\nc(V) \into D_{\dR}(V)$ induced by the injection $\obdr^\nc \into \obdr$  is an isomorphism of $K$-vector spaces.

\item $D_{\dR, K}^\nc(V)$ , $D_{\dR, K}(V)$, $D_{\HT, K}(V)$ are $K$-vector space of dimension $\leq \dim_\Qp V$. 
\end{enumerate}
\end{prop}
\begin{proof}
The statements concerning $\obdr^\nc$  are proved in \cite[Prop. 2.16, Prop. 3.22]{Bri06},  the statements concerning  $\obht$ are mentioned in the proof of \cite[Prop. 3.35]{Bri06}. To finish the proof, it would suffice to show  $D_\dR^\nc(V) \into D_{\dR}(V)$ is an isomorphism (note that the case for $V=\qp$   implies $(\obdr^{\nc})^{G_K}=(\obdr)^{G_K}$).

Now, \cite[Prop. 3.22]{Bri06} \emph{already} shows that $D_\dR^\nc(V)$ is a finite dimensional $K$-vector space, and \cite[Prop. 4.19]{Bri06} \emph{already} shows that the filtration on $D_\dR^\nc(V)$, namely:
$$\fil^i D_\dR^\nc(V): =(V\otimes_\qp \fil^i \obdr^\nc)^{G_K} , \forall i \in \mathbb Z  $$
is exhaustive and separated. Namely, there exists some \emph{finite} numbers $c\leq d$ such that 
\begin{equation*}
D_\dR^\nc(V) = (V\otimes_\qp   \obdr^{\nc, [a,b]})^{G_K}, \forall -\infty \leq a \leq c \leq d \leq b \leq  +\infty. 
\end{equation*}
Note that since $c, d$ are finite, we have 
$$(V\otimes_\qp   \obdr^{\nc, [c, d]})^{G_K} = (V\otimes_\qp   \obdr^{ [c, d]})^{G_K}.$$
An obvious inductive argument then shows 
$$D_\dR (V) = (V\otimes_\qp   \obdr^{  [a,b]})^{G_K}, \forall -\infty \leq a \leq c \leq d \leq b \leq  +\infty. $$
 \end{proof}

\begin{remark}\label{rem223}
The isomorphism  $D_\dR^\nc(V) \simeq D_{\dR}(V)$  shows that both $\obdr^\nc$ and $ \obdr$ can be used to construct the de Rham Fontaine module theory.
\end{remark}

\begin{defn}  Let $V \in \rep_\gk(\qp)$.  
\begin{enumerate}
\item $V$ is a called  de Rham if $D_{\dR, K}^\nc(V)$ (equivalently $ D_{\dR, K}(V)$) is of $K$-dimension $\dim_\Qp V$. 
\item $V$ is called Hodge--Tate if $D_{\HT, K}(V)$ is  of $K$-dimension $\dim_\Qp V$.
\end{enumerate}
 \end{defn}



\subsection{Connection maps}
\begin{notation}
  Let
$$ \wh{\Omega}_\ok:= \projlim_{n >0} \Omega^1_{\ok/\mathbb Z}/p^n\Omega^1_{\ok/\mathbb Z}$$
be the $p$-adically continuous K\"ahler differentials, and let 
$$\wh\Omega_{K}=\wh{\Omega}_\ok\otimes_\ok K.$$ 
$\wh{\Omega}_\ok$ is a free $\ok$-module with a set of basis $\{d\log(t_i)\}_{1\leq i \leq d}$, where $t_1, \cdots, t_d \in \ok$ are introduced in Notation \ref{notatitflat} as a lift of a $p$-basis of $k_K$.
Let 
$$d: \ok \to \wh{\Omega}_{\ok}$$ 
be the induced differential map (extending the canonical differential map $d: \ok \to {\Omega}_{\ok/\mathbb Z}$.
\end{notation}

We now construct  connections on $\obdr^\nc$ and $\obdr$. In \cite[\S 2.2]{Bri06}, a connection is defined in an \emph{explicit} way that seemingly to depend on various choices (such as $u_i$'s). Here, we follow the presentation of \cite[\S 6]{Sch13} to see that   the connection is indeed a \emph{unique} extension of that on $K$, and in particular is an \emph{intrinsic} notion.

\begin{construction} \label{drnabla}
Recall we have $d: \ok \to \wh{\Omega}_{\ok}$. It extends $W(\ocflat)$-linearly  to a \emph{unique} (continuous) connection 
$$\nabla:  W(\ocflat)\otimes_{\mathbb Z}\ok \to W(\ocflat)\otimes_{\mathbb Z}  \wh{\Omega}_{\ok}.$$
Via \eqref{eqainfocok}, it then extends $W(\ocflat)$-linearly (but not $\Ainf(\O_C/\O_{K})$-linearly, since $\nabla$ does not kill $\theta_K^{-1}(p\O_C)$) to
$$\nabla:  \Ainf(\O_C/\O_{K}) \to \Ainf(\O_C/\O_{K}) \otimes_\ok \wh{\Omega}_{\ok},$$
and hence $\bdrplus$-linearly (extending the $W(\ocflat)$-linearity above, via $\ker \theta$-completion) to 
$$\nabla: \obdrplus \to \obdrplus \otimes_K \wh{\Omega}_{K},$$
and finally $\bdr$-linearly (which is possible since $t  \in \bdrplus$) to 
$$\nabla: \obdr^{\nc} \to \obdr^{\nc} \otimes_K \wh{\Omega}_{K}.$$
One can easily check that $\nabla(u_i) = dt_i$, hence indeed this \emph{intrinsically} defined $\nabla$ is exactly the same as the one explicitly defined in \cite{Bri06}. Namely, using the expression $\obdrplus =\bdrplus[[u_1, \cdots, u_d]]$, and for each $i$, let $N_i$ be the unique $\bdrplus$-derivation of $\obdrplus$ such that 
$$N_i(u_j)=\delta_{i, j}t_j$$ 
where $\delta_{i, j}$ is the Kronecker symbol, then $\nabla$ has the following explicit formula:
\begin{equation}
\nabla:\obdr^{\nc} \to \obdr^{\nc}\otimes_K \wh{\Omega}_K, \quad x \mapsto \sum_{i=1}^d N_i(x)\otimes d \log(t_i).
\end{equation} 
We list some basic properties of $\nabla$.
\begin{enumerate}
\item $\nabla$ is integrable (since $N_i$'s commute with each other).
\item We have $(\obdr^{\nc})^{\nabla=0} =\bdr.$ 
\item By \cite[Prop. 2.23]{Bri06}, $\nabla$ satisfies Griffiths transversality, i.e., 
$$\nabla(\Fil^r \obdr^{\nc}) \subset \Fil^{r-1}  \obdr^{\nc} \otimes_K \wh{\Omega}_K.$$
\item By \cite[Prop. 2.24]{Bri06}, $\nabla$ commutes with $G_K$-action.
\item By \cite[Prop. 2.25]{Bri06}, $\nabla|_K$ is precisely the   differential $d: K \to   \wh{\Omega}_K.$ (Note this is now a ``vacuous" statement, as we have just shown $\nabla$ comes from a unique extension of $d$ on $K$).
\end{enumerate}
\end{construction}

\begin{construction}\label{consnablaobdr}
Since $\nabla(t)=0$, the connection $\nabla$ on $\obdr^\nc$ continuously extends to  
$$\nabla:\obdr  \to \obdr \otimes_K \wh{\Omega}_K.$$  
Note $\nabla$ on $\obdr$ still satisfies all the listed  properties above, namely:
\begin{enumerate}
\item $\nabla$ is integrable.
\item   $(\obdr )^{\nabla=0} =\bdr.$ 
\item $\nabla$ satisfies Griffiths transversality, i.e., 
$$\nabla(\Fil^r \obdr ) \subset \Fil^{r-1}  \obdr  \otimes_K \wh{\Omega}_K.$$
\item $\nabla$ commutes with $G_K$-action.
\item   $\nabla|_K$ is precisely the   differential $d: K \to   \wh{\Omega}_K.$
\end{enumerate}
 \end{construction}

 \begin{construction}\label{notagrzero}
 Since $\nabla$ on $\obdr^\nc$ (and $\obdr$) satisfies Griffiths transversality, we can define the zero-th graded:
 $$\gr^0 \nabla: \coc \to \coc(-1)\otimes_K  \wh{\Omega}_K =\coc \otimes_K \wh{\Omega}_K(-1).$$
 Using the expression $\coc \simeq C[\frac{\bar{u}_1}{t}, \cdots, \frac{\bar{u}_d}{t}]$, one sees that $\gr^0 \nabla$ is $C$-linear, and 
 $$\gr^0 \nabla(\frac{\bar{u}_i}{t}) =t_i \otimes \frac{d\log t_i}{t}. $$
 More conveniently, using the expression $\coc  \simeq   C[V_1, \cdots, V_d]$, we have
 $$\gr^0 \nabla(V_i) = 1\otimes \frac{d\log t_i}{t}.$$
 (Let us note that $\nabla$ is \emph{not} $K$-linear, but $\gr^0 \nabla$ is $C$-linear hence $K$-linear.)
 \end{construction}

\section{Change of base fields} \label{secKL}
In this section, we give some examples of the continuous embedding $i: K\into L$, and we set up notations for the field $L$. In the end, we recall Morita's rigidity theorem for a special embedding.

\begin{notation}
\begin{enumerate}
\item Recall in Notation \ref{notakl}, we fixed a continuous field embedding $i: K\into L$. We also fixed $i: \overline K \into \overline L$ and $i: G_L \to G_K$.

\item Throughout this paper, we use $\bbk$ to denote a mixed characteristic CDVF with perfect residue field ($\bbk$ may or may not be related with $K$).
\end{enumerate} 
\end{notation}

\subsection{Embedding of fields}
We introduce three examples of $K \into L$ to orient the readers.

\begin{example}
Let $\qp \into L=W(\fp(x^{\frac{1}{p^\infty}}))[1/p]$ be the inclusion embedding. Note that Thm. \ref{thmintrorigi} implies that a representation $V \in \rep_{G_\qp}(\qp)$  is Hodge--Tate resp. de Rham if and only if $V|_{G_L}$ is so. Note although the two fields both have $p$ as uniformizers, the extension of their residue fields is not algebraic (hence in particular $L/\qp$ is not an unramified extension). Thus any induced inclusion $C_{\qp} \into C_L$ is not an isomorphism, and hence the mentioned  Hodge--Tate resp. de Rham rigidity does not follow from the ``usually known argument".  
\end{example}

\begin{example}
\label{notaspecialbbk}  
In Notation \ref{notatitflat}, for each $m \geq 0$, we fixed the elements $t_i^{\frac{1}{p^m}} \in \overline K$. Let $K^{(\pf)}/K$ be the algebraic extension by adjoining all these $t_i^{\frac{1}{p^m}}$, namely:
$$ K^{(\pf)} =\cup_{m \geq 0} K(t_1^{\frac{1}{p^m}}, \cdots, t_d^{\frac{1}{p^m}}).  $$
 Let $\kpf$ be its $p$-adic completion.   Let 
$$i: K \into\kpf$$ 
be the   inclusion embedding.
\end{example}

\begin{rem}\label{remAdvantagemorita}
Let us summarize some advantageous features of the above embedding, which are crucially exploited in \cite{Bri03} and then used in \cite{MoritadR} to prove Thm. \ref{thmmorita}.
\begin{enumerate}
\item We introduced $u_i, V_i$ in Notations \ref{notatitflat} and \ref{notaVi};  $u_i$ and $tV_i$ are \emph{fixed} by $G_\kpf$.

\item The algebraic extension $K \to  K^{(\pf)}$ has an obvious Galois closure by adjoining all $p$-power roots of unity whose Galois group is an \emph{explicit}  $p$-adic Lie group, cf. Notation \ref{notaGammagp}.
\end{enumerate}
\end{rem}

\begin{example} \label{examBT}
(This example is key for the study of crystalline resp. semi-stable representations, as well as integral $p$-adic Hodge theory in the imperfect residue field case). Now suppose the residue field $k_K$ is not perfect (otherwise, the following discussion is vacuous). The readers can consult \cite[\S 1]{Bri06} and \cite[\S 2]{BT08} for more details.
\begin{itemize}

\item Let $\overline{k_K}$ be the residue field of $\overline{K}$, which is an algebraic closure of $k_K$.
 Let $\mathbf{k}$ be the radical closure of $k_K \subset \overline{k_K}$; note that $\mathbf{k}$ is a perfect field. Let $ W(\mathbf{k})$ be the ring of Witt vectors for $\mathbf k$, recall there is a \emph{unique} (lifted) Frobenius morphism   on $ W(\mathbf{k})$.

\item Let  $C(k_K)$ be a Cohen ring of $k_K$, and \emph{fix} an embedding $C(k_K)\into \ok$  (which is not unique) and denote the image as $\oko \subset \ok$, let $K_0=\oko[1/p]$.

\item   \emph{Fix} a Frobenius  lifting $\sigma: \oko \to \oko$ (such Frobenius liftings are \emph{not} unique). With $\sigma$ fixed, there then exists a \emph{unique} continuous \emph{Frobenius-equivariant}  ring embedding
$$i_\sigma: \oko  \into W(\mathbf{k})$$ 
whose reduction modulo $p$ is the inclusion $k_K \into \mathbf{k}$. 
See \cite[\S 1, paragraph 2]{Bri06} for full details.


\item Let $\bbk_0 =W(\mathbf{k})[1/p]$, and \emph{fix} $\overline{\bbk_0}$ an algebraic closure of $\bbk_0$. 
The map $i_\sigma$   extends (non-uniquely) to some embedding 
$$i_\sigma: \overline{K} \into \overline{\bbk}.$$ 

\item Finally, define $\bbk_\sigma:=i_\sigma(K)\bbk_0$. Then we obtain an embedding 
$$i_\sigma: K \into \bbk_\sigma.$$
\end{itemize}
Note that $\bbk_\sigma$ is always isomorphic to $\kpf$ in Notation \ref{notaspecialbbk} as they are both the ``smallest" CDVF containing $K$ whose residue field is the perfection of $k_K$. However, it is clear the embeddings $i_\sigma$ can be very different from the inclusion $K \into \kpf$.
\end{example}

\begin{rem}
Clearly, the features mentioned in 
 Rem. \ref{remAdvantagemorita} in general do not apply in Example \ref{examBT}.
\end{rem}

\subsection{Fontaine rings for different base fields}
\begin{notation}
Recall in Notation \ref{notakl}, we fixed a continuous embedding $i: K \into L$.
We also fixed $i: \overline K \into \overline L$ and $i: G_L \to G_K$.
Let $C_L$ be the $p$-adic completion of $\overline L$, then we have an induced embedding $i: C \into C_L$ that is compatible with the Galois actions via $i: G_L \to G_K$.
As $L$ is also a CDVF whose residue field has finite $p$-basis, we can carry out all the constructions in Section \ref{secfontaine}. Indeed, we can first define the rings   $\mathcal{O}_{C_L}^\flat$, $W(\mathcal{O}_{C_L}^\flat)$; note that $i$ induces continuous homomorphisms $\ocflat \to \mathcal{O}_{C_L}^\flat$, $W(\ocflat) \to W(\mathcal{O}_{C_L}^\flat)$ that are compatible with Galois actions via $i: G_K \to G_L$.
Note that since $i: K \into L$ is continuous, there exists some $s \in \mathbb Z$ such that $i(\ok) \subset p^s\calO_L$. This implies that 
$K \into L$ can be extended to a  continuous  homomorphism
$$i: \Ainf(\O_C/\O_K) \otimes_{\Zp} \Qp \to \Ainf(\O_{C_L}/\O_L) \otimes_{\Zp} \Qp$$
and then to 
$$i: \obdrplus \to \calO \mathbf{B}_{\dR, L}^{+}.$$ 
There are similar continuous homomorphisms to $\O\mathbf{B}_{\dR, L}, \O \mathbf{B}_{\HT, L}, \O C_L$ etc., that are always compatible with structures such as Galois actions, filtrations, connections (whenever applicable).
\end{notation}



\begin{notation}\label{notabbkperf}
For $\bbk$ a  CDVF with perfect residue field ($\bbk$ may or may not be related with $K$), we fix an algebraic closure $\overline{\bbk}$ and its $p$-adic completion $\bbc$. We then use notations such as $$\bbdr^+,\quad \bbdr,\quad \bbht $$ to denote the ``usual" Fontaine  rings (in the perfect residue field case).
\end{notation}

\subsection{Morita's  results} 
 The following theorem of Morita is a special case of our Thm. \ref{thmintrorigi}(2).
 Morita's argument relies on the features mentioned in Rem. \ref{remAdvantagemorita}, and cannot be adapted to the general case.
 
\begin{thm} \label{thmmorita} \cite{MoritadR}  Let $V\in \rep_\gk(\qp)$.  Let $K \into \kpf$ be the embedding in Example \ref{notaspecialbbk}.
Then $V$ is Hodge--Tate (resp. de Rham) if and only if $V|_{G_{\kpf}}$ is  so.
\end{thm} 

\begin{remark}
In \cite{Moritacrys}, Morita also proves that $V$ is  potentially semi-stable (resp. potentially crystalline) if and only if  $V|_{G_{\kpf}}$ is  so. This, together with Thm. \ref{thmmorita}, allows him to obtain  \emph{$p$-adic local monodromy theorem} in the imperfect residue field case: namely, $V$ is de Rham if and only it is  potentially semi-stable. Alternatively, another proof which works even when $[k:k^p]=+\infty$ is supplied by Ohkubo \cite{Ohk13}.
\end{remark}

\section{Sen theory and Tate--Sen decompletion} \label{secTateSen}
In this section, we review   Sen theory, first in the classical perfect residue field case, then in the imperfect residue field case. We also carry out a Tate--Sen decompletion in the imperfect residue field case, and use it to carry out some cohomological computations for later use.

\subsection{Sen theory in the perfect residue field case} \label{subsec41}
The following theorem is proved in \cite{Sen81}, with some key ingredients from \cite{Tat67}. We also refer to \cite{Fon04} for a very clear exposition.  Let $\bbk$ be a CDVF with perfect residue field as in Notation \ref{notabbkperf}. Let $\bbkinfty$ be the extension by adjoining all $p$-power roots of unity, and let $\hatbbkinfty$ be the $p$-adic completion. Let
$G_{\bbkinfty}: =\gal(\overline{\bbK}/\bbkinfty),  \Gamma_\bbK :=\gal(\bbkinfty/\bbk)$.

\begin{prop}
The rule 
$$\rep_{G_\bbk}(\bbc) \ni W \mapsto \calH_\bbk(W):= W^{G_{\bbkinfty}}$$
 induces a tensor equivalence of categories
$$\calH_\bbk: \rep_{G_\bbk}(\bbc) \xrightarrow{\simeq}  \rep_{\Gamma_\bbk}(\hatbbkinfty).$$
The base change map $$\rep_{\Gamma_\bbk}(\bbkinfty)                                                     \ni M \mapsto M\otimes_\bbkinfty \hatbbkinfty$$
 induces a tensor equivalence of categories
$$ \rep_{\Gamma_\bbk}(\bbkinfty)                                                      \xrightarrow{\simeq}   \rep_{\Gamma_\bbk}(\hatbbkinfty). $$
These equivalences then induce a tensor equivalence, denoted as $H_\bbk$:
\[\mathrm{H}_\bbk: \rep_{G_\bbk}(\bbc) \xrightarrow{\simeq}   \rep_{\Gamma_\bbk}(\bbkinfty).                                                 \]
\end{prop}
 
Let us gather some results used in the proof of above theorem.


\begin{prop}\label{thmFon0411}
\begin{enumerate}
\item $H^1(G_{\bbkinfty}, \GL_n(\bbC)) =0$ for  any $n\geq 1$.

  \item Any $W \in \rep_{G_{\bbkinfty}}(\bbC)$ is trivial.

\item Let $W \in \rep_{G_{\bbkinfty}}(\bbC)$, then $H^1(G_{\bbkinfty}, W)=0$. (As a special case, $H^1(G_{\bbkinfty}, \bbC)=0$).
\end{enumerate} 
\end{prop}
\begin{proof}
See \cite[Thm. 1.1, Cor 2.3]{Fon04}.
\end{proof}

\begin{cor}\label{thmalmostvanish}
\begin{enumerate}
\item $H^i(G_{\bbkinfty}, \bbc) =0, \forall i \geq 1$.
\item Let   $W \in \rep_{G_{\bbkinfty}}(\bbC)$, then $H^i(G_{\bbkinfty}, W) =0, \forall i \geq 1$.
\end{enumerate}
\end{cor}
\begin{proof}
Note that $G_{\bbkinfty}\simeq G_{\hatbbkinfty}$ where the later denotes the Galois group of the perfectoid field $\hatbbkinfty$. Then Item (1) follows from the (almost) vanishing theorem of \cite[Prop. 7.13]{Sch12}. For Item (2), $W$ is a trivial representation by Prop. \ref{thmFon0411}, hence the statement follows from Item (1).
\end{proof}

\subsection{Brinon's Sen theory in the imperfect residue field case} \label{subsecBrinonSen}

We now recall Brinon's Sen theory in the imperfect residue field case, cf. \cite{Bri03}. 

\begin{notation}
We  introduce the following diagram of field extensions.
\begin{equation}
\begin{tikzcd}
\overline{K} \arrow[rr, hook]                                   &  & \overline{K^\pf} \arrow[rr, hook]                            &  & \overline{\wh{K^\pf_\infty}} \arrow[rr, "\text{completion}", hook] &  & C \\
K^{(\pf)}_\infty \arrow[u, "H_{\underline t}"] \arrow[rr, hook] &  & K^\pf_\infty \arrow[rr, "\text{completion}", hook] \arrow[u] &  & \wh{K^\pf_\infty} \arrow[u]                                        &  &   \\
K^{(\pf)} \arrow[u] \arrow[rr, "\text{completion}", hook]       &  & K^\pf \arrow[u]                                              &  &                                                                    &  &   \\
K \arrow[u] \arrow[uu, "\Gamma", bend left=49]                  &  &                                                              &  &                                                                    &  &  
\end{tikzcd}
 \end{equation}
 Let us explain the notations:
 \begin{enumerate}
 \item The general pattern of the diagram is that all vertical arrows are \emph{algebraic} extensions, and all horizontal arrows are \emph{dense} embeddings. All fields  on the right most ($K, \kpf, \wh{K^\pf_\infty}, C$) are $p$-adically complete.
 \item Fields on the third row are introduced in  Example \ref{notaspecialbbk}.
 \item On the second row: $K^{(\pf)}_\infty$ resp. $K^\pf_\infty$ are obtained from $K^{(\pf)}$ resp. $\kpf$ by adjoining all $p$-power roots of unity; $\widehat{K^\pf_\infty}$ is the $p$-adic completion.
 
 \item Fields on the first row (except $C$) are algebraic closures of fields on the second row; $C$ is the $p$-adic completion of all fields on this row.
 
 \item We also introduce the Galois groups $H_{\underline{t}}:= \gal(\overline{K}/K^{(\pf)}_\infty)$ and $\Gamma:= \gal(K^{(\pf)}_\infty/K)$.
 \end{enumerate}
\end{notation}

\begin{prop}
\cite[Thm. 1, Thm. 2]{Bri03}\label{thmBrinonSen}
\begin{enumerate}
\item $C^{H_{\underline{t}}} =\wh{K^\pf_\infty}$.
\item For $W \in \rep_{G_K}(C)$,  $W^{\Ht}$ is an object in  $\rep_\Gamma(\widehat{K_\infty^\pf})$, and there is a $G_K$-equivariant isomorphism 
$$W^\Ht \otimes_{\widehat{K_\infty^\pf}} C \simeq W. $$
The rule $W   \mapsto W^\Ht$ induces a tensor equivalence of categories:
$$\rep_{G_K}(C) \simeq \rep_\Gamma(\widehat{K_\infty^\pf}).$$

\item The rule 
$$
 \rep_\Gamma( {K_\infty^{(\pf)}}) \ni M  \mapsto  M \otimes_{ {K_\infty^{(\pf)}}}  \widehat{K_\infty^\pf} \in \rep_\Gamma(\widehat{K_\infty^{\pf}}) $$
   induces a tensor equivalence  of categories:
$$\rep_\Gamma( {K_\infty^{(\pf)}}) \simeq \rep_\Gamma(\widehat{K_\infty^{\pf}}).$$  
\end{enumerate} 
\end{prop} 

\begin{notation}\label{notaGammagp}
Let $\kinfty =\cup_{n \geq 0} K(\varepsilon_n)$, then we have the following diagram of algebraic extensions, where the groups over the arrows denote Galois groups:
\begin{equation}
\begin{tikzcd}
                                   & K^{(\pf)}_\infty                    &                                             \\
\kinfty \arrow[ru, "\Gamma_\geom"] &                                     & K^{(\pf)} \arrow[lu, "\Gamma_{K^{(\pf)}}"'] \\
                                   & K \arrow[lu, "\Gamma_K"] \arrow[ru] &                                            
\end{tikzcd}
\end{equation}
By Lem. \ref{lem: intersect kinfty kpf} in the appendix (cf. also Rem. \ref{rem: kinfty not tot ram}), 
there is an isomorphism
$$\Gamma_{K^{(\pf)}} \simeq \gammak.$$
There is a short exact sequence
$$ 1\to \Gamma_\geom \to \Gamma \to \Gamma_K \to 1.$$
Note $\Gamma_\geom :=\gal( K^{(\pf)}_\infty/ \kinfty) \simeq \zp(1)^d$, where $\zp(1)^d$ signifies that $\Gamma_{K^{(\pf)}}$ acts on $\Gamma_\geom$ via $\chi_p^{\oplus d}$ where $\chi_p$ denotes the $p$-adic  cyclotomic character. Thus, there is an isomorphism
$$   \Gamma_{K^{(\pf)}} \ltimes \Gamma_\geom \simeq \Gamma.$$
Note all the three groups above are $p$-adic Lie groups, and we use $\mathrm{Lie} \Gamma$ etc. to denote their Lie algebras.
\end{notation}

 \begin{prop}\label{thmnilpotent}
 \cite[Prop. 5]{Bri03} (cf. also the summary in \cite[Lem. 4.4]{Ohk11}).
Let $M \in \rep_\Gamma( {K_\infty^{(\pf)}})$.
\begin{enumerate}
\item The action of $\Gamma_{K^{(\pf)}}$  on $M$ is \emph{locally analytic} (cf. \S\ref{subsubLAV}), in the sense that there exists a $K^{(\pf)}_\infty$-linear endomorphism $\varphi$ of $M$   such that for any $y \in M$, there exists an open subgroup $H(y) \subset \Gamma_{K^{(\pf)}}$ such that 
$$g(y) = \mathrm{exp}(\varphi \log \chi_p(g)) y, \forall g\in H(y). $$
Indeed, the $\varphi$-operator is nothing but the  action of (the 1-dimensional Lie algebra) $\Lie(\Gamma_{K^{(\pf)}})$.

\item The action of the (commutative Lie-group) $\Gamma_\geom$ on $M$ is \emph{locally analytic}, in the sense that there exist  $K^{(\pf)}_\infty$-linear endomorphisms $\mu_1, \cdots, \mu_d$ of $M$ such that for  any $y \in M$, there exists an open subgroup $H(y) \subset \Gamma_\geom$ such that 
$$g(y) = \mathrm{exp}(\sum_{i=1}^d e_i\mu_i) y, \forall g = \prod_{i=1}^d \tau_i^{e_i} \in H(y). $$
Indeed, these $\mu_i$-operators come from the  action of (the d-dimensional Lie algebra) $\Lie(\Gamma_\geom)$.

\item Furthermore, the endomorphisms $\mu_1, \cdots, \mu_d$ are \emph{nilpotent}.
\end{enumerate}
 \end{prop}

 
 \subsection{Tate--Sen decompletion in the imperfect residue field case} \label{subsec43}

 
 \begin{notation} \label{nota: assume intersect}
Use Notation \ref{notaGammagp}.
Assume:
\begin{itemize}
\item $K$ contains all $p^2$-th roots of unity; 
\end{itemize}
Under above assumption,  we have
\begin{equation} \label{eq: equal gamma gps}
\Gamma_{K^{(\pf)}} \simeq \Gamma_{K^{\pf}} \simeq \gammak \simeq \zp.
\end{equation}
 Use $\gamma_0$ to denote a topological generator of these groups.
For each $s \geq 0$, let $K_s \subset \kinfty$ be the unique sub-extension such that $[K_s:K]=p^s$. 
For $x\in \kinfty$, define Tate's normalized traces
$$t_{K_s}(x) :=p^{-n+s} \mathrm{Tr}_{K_n/K_s}(x) \in K \text{ for $n \gg 0$  such that $x \in K_n$} $$
Similarly define $K^\pf_s$ and $t_{K^\pf_s}$. It is clear that these trace maps are compatible (for elements in $\kinfty$), namely,
\begin{equation}\label{eqcompakkpf}
t_{K_s}(x)=t_{K^\pf_s}(x), \forall x\in \kinfty
\end{equation}
\end{notation} 
 

 \begin{prop}\label{thm432imperf}
Suppose   assumption  in Notation \ref{nota: assume intersect} is satisfied. 
 For any $h \geq 1$, the map 
\begin{equation} \label{tsimperf}
\iota: H^1(\gammak, \GL_h(\kinfty)) \to  H^1(\gammak, \GL_h(\hatkinfty))
\end{equation}
induced by $\GL_h(\kinfty) \into \GL_h(\hatkinfty)$ is a bijection.
 \end{prop}
\begin{proof}
When $K$ has perfect residue field, this is a theorem of Sen (cf. \cite[Thm. 1.2']{Fon04}). Thus, we have 
\begin{equation} \label{tsperf}
\iota_\kpf: H^1(\Gamma_\kpf, \GL_h(K^\pf_\infty)) \to  H^1(\Gamma_\kpf, \GL_h(\widehat{K^\pf_\infty})) \text{ is a bijection. }
\end{equation}
We now show that indeed all the ingredients leading to the proof of \eqref{tsperf} still hold for $K$, and hence the same strategy can be used to prove \eqref{tsimperf} is bijective. We will follow the  exposition in the proof \cite[Thm. 1.2']{Fon04}.

\textbf{Ingredient 1} (cf. \cite[Prop. 1.13]{Fon04}): For $x\in \kinfty$, we have
\begin{equation}\label{eq113fon04}
v_p(t_K(x) -x) \geq v_p( (\gamma_0 -1)x) -\frac{p}{p-1}
\end{equation}
Note this is proved for $x \in K^\pf_\infty$ (and using $t_\kpf$) in \cite[Prop. 1.13]{Fon04}, hence this \emph{automatically} holds for $x\in \kinfty$ by the compatibility \eqref{eqcompakkpf}.

\textbf{Ingredient 2} (cf. \cite[Prop. 1.15]{Fon04}): the following statements hold.
\begin{enumerate}
\item $t_K: \kinfty \to K$ is continuous.
\item Let $\hat{t}_K: \hatkinfty \to K$ be the continuous extension of $t_K$, and let $L_0$ denote the kernel. Then we have $\hatkinfty =K \oplus L_0$. Furthermore, $\gamma_0 -1$ is bijective on $L_0$ with a continuous inverse $\rho$.
\item $v_p(\hat{t}_K(x)) \geq v_p(x) -\frac{p}{p-1}$ for all $x\in \hatkinfty$ and $v_p(\rho(y))  \geq v_p(y) -\frac{p}{p-1}$ for all $y \in L_0$.
\item For all $x\in \hatkinfty$,  we have $\lim_{s \to \infty} \hat{t}_{K_s}(x) =x$
\end{enumerate}
Item (1), (4), as well as the first inequality of Item (3) follow  by \eqref{eqcompakkpf} and the known result for $\kpf$. It suffices to construct $\rho$   in Item (2), and prove the second inequality of Item (3). 
The construction of $\rho$ does not directly follow from the $\kpf$-case, but the exact same argument works, and let us give a quick sketch. 
(We also find the original argument in  \cite[p173, Prop. 7]{Tat67} very concise.)
Indeed, note that $L_0=(1-\hat{t}_K)\hatkinfty$.
Now let $K_{n, 0} =K_n \cap L_0$, and let $K_{\infty, 0} =\cup_{n\geq 0} K_{n, 0}$, then its closure is $L_0$. Note $\gamma_0-1$ is injective and $K$-linear on the \emph{finite dimensional} $K$-vector spaces $K_{n, 0}$, and hence is bijective. 
Thus $\gamma_0-1$ is bijective on $K_{\infty, 0}$; let 
$$\rho: K_{\infty, 0} \to K_{\infty, 0}$$
 be its inverse. 
Let $x \in   K_{\infty, 0}$ (hence $t_K(x)=0$), let $y=(\gamma_0 -1)x$, hence $\rho(y)=x$. Thus Eqn. \eqref{eq113fon04} in Ingredient 1 implies that  
\begin{equation}\label{eqrhoy}
v_p(\rho(y)) \geq v_p(y) -\frac{p}{p-1}. 
\end{equation}
Thus $\rho$ is continuous (as $y$ runs through all of $ K_{\infty, 0}$) and extends to $L_0$, and the second inequality of Item (3) follows from \eqref{eqrhoy}.

\textbf{Ingredient 3}: 
A finite dimensional $K$-sub vector space contained in $\hatkinfty$, which is stable by $\gamma_0$, is contained in $\kinfty$.
The proof of \cite[Prop. 1.16]{Fon04} works verbatim here, since we have Ingredient 2.

With all three ingredients established, all the argument below  \cite[Prop. 1.16]{Fon04}  work through verbatim, which establishes \eqref{tsimperf}.
\end{proof}

 \begin{rem}\label{remts123}
 Note that Ingredients 1 and 2 above show  that   if we use
$$G_0 =\gammak,  \quad  \wt{\Lambda} =\hatkinfty, \quad  \Lambda_{n}=K_n, \quad  R_n =\hat{t}_{K_n}$$
then these data satisfy all the axioms (TS1), (TS2), (TS3) in \cite[Def. 3.1.3]{BC08}. Note that in this case, the kernel (denoted by $H_0$ in \emph{loc. cit.}) of $\gammak \to \zp^\times$ is trivial, hence (TS1) of \emph{loc. cit.} becomes \emph{vacuous}, and hence we can choose any $c_1>0$. We can then choose $c_2=c_3=\frac{p}{p-1}$.
 \end{rem}

\begin{prop}\label{thmfon24}
For $X\in \rep_{\Gamma_K}(\hatkinfty)$, let $X_f$ be the union of finite dimensional $K$-sub-vector spaces of $X$ that are stable under $\gammak$, then  $X_f\in \rep_{\Gamma_K}(\kinfty)$, and $X_f \otimes_\kinfty \hatkinfty \to X$ is a canonical isomorphism.
\end{prop}
\begin{proof}
 If $K$  contains all $p^2$-th roots of unity, then  the argument in  \cite[Thm. 2.4]{Fon04}  works verbatim, as now Prop. \ref{thm432imperf} is established. 
The general case follows easily: indeed, simply consider $K'/K$ the extension by adjoining all $p^2$-th roots of unity. Restrict $X$ as   representation of $\Gamma_{K'}$, and let $X_f'$ be the union of finite dimensional $K$-sub vector spaces of $X$ that are stable under  $\Gamma_{K'}$. Now the key point is to note we must have $X_f=X_f'$ as $\Gamma_{K'} \subset \gammak$ is of finite index, hence $X_f'$ is indeed automatically $\gammak$-stable.
\end{proof}

\begin{cor}\label{corequivkinfty}
For $Y \in \rep_{\Gamma_K}(\kinfty)$, the extension $Y \otimes_\kinfty \hatkinfty$ is an object in $\rep_{\Gamma_K}(\hatkinfty)$.
This induces a tensor equivalence of categories: 
$$\rep_{\Gamma_K}(\kinfty) \xrightarrow{\simeq } \rep_{\Gamma_K}(\hatkinfty).$$
\end{cor}

\subsection{On $\gammak$-cohomology}


\begin{lemma}\label{lemliuzhu310}
Let $M \in \rep_\gammak(K_{m_0})$ for some $m_0 \geq 0$. If $m \gg m_0$, then for all $\gamma \in \gammak$ such that $v_p(\chi_p(\gamma)-1) >m$, $\gamma-1 $ is continuously invertible on $(M\hat{\otimes}_{K_{m_0}} \hatkinfty)/(M\otimes_{K_{m_0}} K_m)$. Thus, for $i \geq 0$, we have natural isomorphisms
$$H^i(\gammak, M\otimes_{K_{m_0}} K_m) \simeq H^i(\gammak, M\hat{\otimes}_{K_{m_0}} \hatkinfty)  $$
\end{lemma}
\begin{proof}
The proof is exactly the same as \cite[Lem. 3.10]{LZ17}, since we already constructed Tate's normalized trace in this setting.
\end{proof}


\begin{cor}\label{corhiggsbasechange}
Let $\hat{M} \in \rep_\gammak(\hatkinfty)$, then 
\begin{enumerate}
\item  $H^i(\gammak, \hat{M})$ is a finite dimensional $K$-vector space for $i \geq 0$, and vanishes when $i\geq 2$.
\item Let $K \into L$ be the field embedding as in Notation \ref{notakl}.
For each $i \geq 0$, the natural map
$$ H^i(\gammak, \hat{M})\otimes_K L \to H^i(\Gamma_L, \hat{M}\otimes_\hatkinfty \hatlinfty)$$
is an isomorphism.
\end{enumerate}
\end{cor}
\begin{proof}
Consider Item (1).
By Cor. \ref{corequivkinfty}, there exists some $M \in \rep_\gammak(\kinfty)$ and hence some $M_{m_0} \in \rep_\gammak(K_{m_0})$ whose base change to $\hatkinfty$ is $\hat{M}$. By Lem. \ref{lemliuzhu310}, $H^i(\gammak, \hat{M}) =H^i(\gammak, M_{m_0})$ and hence is of finite $K$-dimension; they vanish for $i \geq 2$ since $\Gamma_K$ has cohomological dimension $1$.

Consider Item (2).
By Lem. \ref{lemliuzhu310}, it suffices to show for   $m \gg m_0$, the natural map
\begin{equation}\label{eqgakl}
H^i(\gammak, M_m)\otimes_K L \to  H^i(\Gamma_{L},  M_m \otimes_{K_m} L_m) =H^i(\Gamma_{L},  M_m \otimes_{K} L)
\end{equation}
is an isomorphism. 
Consider when $i=1$. Note the natural map $\Gamma_L \to \Gamma_K$ is an open embedding; thus Hochschild--Serre spectral sequence  implies
\[ H^1(\Gamma_L, M_m)\simeq H^1(\gammak, M_m)\]
since finite group cohomology (for $\Gamma_K/\Gamma_L$) on a $\qp$-vector space is concentrated in degree zero; cf. \cite[Cor. 16.5]{Hilton_Stammbach_homological_GTM_v2}.
Thus  flat base change implies:
\[H^1(\gammak, M_m)\otimes_K L \simeq H^1(\Gamma_{L},  M_m \otimes_{K} L).\] 
It remains to treat the case $i=0$. First,  for any $s\geq 0$, by Galois descent, we have
\begin{equation}\label{eqHiks}
H^0(\gammak, M_m)\otimes_K K_s \simeq H^0(\Gamma_{K_s},  M_m), \forall m \geq s.
\end{equation}
Now, fix some $s \gg 0$ so that the map $\Gamma_{L_s} \into \Gamma_{K_s}$ is an isomorphism.
To prove Eqn. \eqref{eqgakl} is an isomorphism, it suffices to show
$$H^0(\gammak, M_m)\otimes_K L \otimes_L L_s \simeq  H^0(\Gamma_{L},  M_m \otimes_{K_m} L_m)\otimes_L L_s, \forall m \gg s.$$
Using \eqref{eqHiks}, it is the same to show
$$H^0(\Gamma_{K_s},  M_m)\otimes_{K_s} L_s \simeq  H^0(\Gamma_{L_s},  M_m \otimes_{K_m} L_m),$$
which is
$$H^0(\Gamma_{K_s},  M_m)\otimes_{K } L  \simeq  H^0(\Gamma_{L_s},  M_m \otimes_{K } L ).$$
But the above holds by flat base change.
\end{proof}



\section{$p$-adic Simpson correspondence and Hodge--Tate rigidity}
\label{secSimpson}
In this section, we construct the $p$-adic Simpson correspondence together with its decompleted version, and use it to prove a Hodge--Tate rigidity theorem.
The following diagram summarizes the (base) fields and Galois groups that we use in this section.

 \begin{equation}
\begin{tikzcd}
\hatkinfty \arrow[rr]                                                                                                                 &  & \widehat{K_\infty^{\pf}} \arrow[rr]                         &  & C                      \\
                                                                                                                                      &  &                                                             &  &                        \\
\kinfty \arrow[rr, "\Gamma_{\mathrm{geom}}"] \arrow[rrrr, "\quad \quad \quad \quad \quad  \quad   G_{\kinfty}", bend left] \arrow[uu] &  & K_\infty^{(\pf)} \arrow[rr, "H_{\underline{t}}"] \arrow[uu] &  & \overline K \arrow[uu] \\
                                                                                                                                      &  &                                                             &  &                        \\
K \arrow[uu, "\Gamma_K"] \arrow[rruu, "\Gamma"] \arrow[rrrruu, "G_K"']                                                                &  &                                                             &  &                       
\end{tikzcd}
 \end{equation}

 \subsection{$p$-adic Simpson correspondence in the imperfect residue field case}

 
 \begin{defn}
 Let $\mathrm{Higgs}_\gammak(\hatkinfty)$ be  the category where an object is some $M \in \rep_{\Gamma_K}(\hatkinfty)$ equipped with a $\gammak$-equivariant Higgs field:
$$\theta_M:  M \to M \otimes_K \wh{\Omega}_K(-1).$$
 \end{defn}

\begin{theorem}\label{thmSimpson}
For $ W\in \rep_{G_K}(C)$, define 
$$  \calH(W):= (W\otimes_C \calO C)^{G_{\kinfty}}$$
\begin{enumerate}
\item The above rule defines a functor 
$$\calH: \rep_{G_K}(C) \to  \rep_{\Gamma_K}(\hatkinfty).$$
Via the equivalence in Cor. \ref{corequivkinfty}, we can define a   functor 
$$\mathrm{H}:   \rep_{G_K}(C) \to  \rep_{\Gamma_K}(\kinfty). $$

\item  Using the field $L$, we can similarly define functors 
\begin{align*}
\calH_L: & \rep_{G_L}(C_L) \to \rep_{\Gamma_L} (\widehat{L_\infty}) \\
\mathrm{H}_L: & \rep_{G_L}(C_L) \to \rep_{\Gamma_L} ({L_\infty}).
\end{align*}
   Note that for $W \in \rep_{G_K}(C)$, the base change  $W\otimes_C C_L$ can be regarded as an object in  $\rep_{G_L}(C_L)$. With these notations, we have the following canonical ``base change" isomorphisms:
   \begin{align*}
   \calH(W) \otimes_{\hatkinfty} \widehat{L_\infty} & \simeq \calH_L(W\otimes_C C_L), \\
  \mathrm{H}(W) \otimes_{\kinfty} L_\infty  &\simeq \mathrm{H}_L( W\otimes_C C_L).
   \end{align*}


\item The functor $\calH$ can be upgraded to a functor 
$$\rep_{G_K}(C) \to \mathrm{Higgs}_\gammak(\hatkinfty),$$
in the sense its composite with the forgetful functor $ \mathrm{Higgs}_\gammak(\hatkinfty) \to \rep_{\Gamma_K}(\hatkinfty)$  recovers $\calH$.

\item The functors $\calH$ (as well as its upgraded version) and $\mathrm{H}$ are tensor functors and are compatible with duality.

\end{enumerate} 
\end{theorem} 
 
 \begin{rem}
 By abuse of terminology, we call both $\calH$ and its upgraded version the $p$-adic Simpson correspondence. (Indeed, the Higgs field on $\calH(W)$ does not play any further role in this paper.) We call the functor $\mathrm{H}$ the decompleted  $p$-adic Simpson correspondence.
 \end{rem}
 
 \begin{proof}
 Item (1) will be proved in Thm. \ref{thm514}, where we also prove some cohomology vanishing results for later use; Items (2)-(4) are formal consequences of   Item (1), and will be proved after Thm. \ref{thm514}. 

Let us now sketch the ideas of the proof of Item (1), which mimics the strategy of \cite{LZ17}.
 The main content of the  proof can be summarized in the following diagram
 \begin{equation}
\begin{tikzcd}
\rep_{G_K}(C) \arrow[rr, "\simeq"] &  & \rep_{\Gamma}(\widehat{K^\pf_\infty}) \arrow[rr, "\simeq"] \arrow[dd, "(\cdot)^{\Gamma_{\geom}-\text{unip}}"'] &  & \rep_{\Gamma}({K^{(\pf)}_\infty}) \arrow[dd, "(\cdot)^{\Gamma_{\geom}-\text{unip}}"'] \\
                                   &  &                                                                                                               &  &                                                                                       \\
                                   &  & \rep_{\Gamma_K}(\hatkinfty) \arrow[rr, "\simeq"]                                                              &  & \rep_{\Gamma_K}(\kinfty)                                                            
\end{tikzcd}
\end{equation}
 Here, the top horizontal equivalences are due to Brinon as in Thm. \ref{thmBrinonSen}. The vertical functors are defined by ``taking unipotent part under the $\Gamma_\geom$-action". A natural question  here  is that if the composite functor $\rep_{G_K}(C) \to \rep_{\Gamma_K}(\hatkinfty)$ depends on the choices of $t_i$: indeed it does not, as answered in Step 2 in the proof of Thm. \ref{thm514}.
 \end{proof}

We start with a lemma.
\begin{lemma}\label{lemSimpcoho}
Let $W \in \rep_{G_K}(C)$. 
\begin{enumerate}
\item   $H^i(\Ht, W\otimes_C \coc)=0$ for $i \geq 1$.
\item    $H^i(G_{\kinfty}, W\otimes_C \coc) =H^i(\Gamma_\geom, (W\otimes_C \coc)^\Ht)$ for $i \geq 0$.
\end{enumerate}
\end{lemma}
\begin{proof}
 For Item (1), note that $\coc =C[V_1, \cdots, V_d]$ and  $\Ht$ acts on $V_i$ trivially, hence   it suffices to show $H^i(\Ht, W)=0$. This follows from Cor. \ref{thmalmostvanish} since  $\Ht \simeq G_{K^\pf_\infty}$ (and $\kpf$ has perfect residue field).  Item (2) now follows by Hochschild--Serre spectral sequence.
\end{proof}

\begin{theorem}\label{thm514}
Let $W \in \rep_{G_K}(C)$. Then $\calH(W): =(W\otimes_C \calO C)^{G_{\kinfty}}$ is an object in $ \rep_{\Gamma_K}(\hatkinfty)$ and the natural $G_K$-equivariant map
\begin{equation} \label{eqchcoc}
\calH(W)\otimes_\hatkinfty \coc \to W\otimes_C \coc
\end{equation} 
is an isomorphism.
In addition,
\begin{equation}
H^i(G_{\kinfty}, W\otimes_C \calO C) =0, \forall   i \geq 1.
\end{equation}  
\end{theorem} 
 \begin{proof}
 
  In Step 1, we first construct a $\hatkinfty$-vector space $\calH$ (of correct dimension). In Step 2, we show it is isomorphic to the desired $\calH(W)$  \emph{as a vector space}; note that $\calH(W)$ is automatically equipped with a $\gammak$-action by definition.

 \textbf{Step 1:}
  Let  
  $$\cm \in \rep_{\Gamma}({\widehat{K^\pf_\infty}}), \quad  \text{ resp.} \quad M \in  \rep_{\Gamma}({{K^{(\pf)}_\infty}} )$$
   be the  objects corresponding to $W$ via Prop. \ref{thmBrinonSen}.
Recall ${{K^{(\pf)}_\infty}}= K(1^{\frac{1}{p^\infty}}, t_1^{\frac{1}{p^\infty}}, \cdots, t_d^{\frac{1}{p^\infty}}) $. Since $\Gamma$ is finitely generated, there exists some $m\gg 0$ such that there exists some $M_m \in \rep_{\Gamma}(K(1^{\frac{1}{p^m}}, t_1^{\frac{1}{p^m}}, \cdots, t_d^{\frac{1}{p^m}}))$ such that its base change   to  ${{K^{(\pf)}_\infty}}$ is $M$.

By Prop. \ref{thmnilpotent}(3), the $K$-linear $\Gamma_\geom$-action on $M_m$ is \emph{quasi-unipotent}.
Hence by \emph{exactly the same argument} as below \cite[Lem. 2.15]{LZ17},  we have a decomposition
\begin{equation}\label{eq215lz}
M_m =\bigoplus_\tau M_{m, \tau}
\end{equation}
where $\tau$ runs through characters of $\Gamma_\geom$ of finite order, and  
$$M_{m, \tau} =\{ m\in M_m | \quad \exists N \gg 0, (\gamma -\tau(\gamma))^Nm=0, \forall \gamma \in \Gamma_\geom \} $$
 is the corresponding generalized eigenspace.
 Now \emph{exactly the same argument} as in \emph{ibid.} shows that if we let $H_m := M_{m, \mathbf{1}}$ be the unipotent part (where $\mathbf{1}$ denotes the trivial character), then $H_m$ is a finite dimensional $K_m$-vector space stable under $\Gamma$-action such that there is a $\Gamma$-equivariant isomorphism
 $$H_m \otimes_{K_m}   K(1^{\frac{1}{p^m}}, t_1^{\frac{1}{p^m}}, \cdots, t_d^{\frac{1}{p^m}}) \simeq M_m.$$ 
Define 
$$\calH: =H_m \otimes_{K_m} \hatkinfty.$$
By construction, there is a $\Gamma$-equivariant isomorphism
 \begin{equation}
 \calH \otimes_\hatkinfty {\widehat{K^\pf_\infty}} \simeq \cm,
 \end{equation}
 and hence there is a
  $G_K$-equivariant   isomorphism
 \begin{equation} \label{eqcalhW}
 \calH \otimes_\hatkinfty C \simeq W.
 \end{equation}

 

\textbf{Step 2.} In order to finish the proof, it suffices to prove there are $\hatkinfty$-linear isomorphisms
\begin{equation} \label{eqstep2}
H^i(G_{\kinfty}, W\otimes_C \calO C) \simeq \begin{cases} \calH  & i=0 \\ 0 & i>0,\end{cases}
\end{equation} 

 
  First note by Lem. \ref{lemSimpcoho}, 
\begin{equation*}\label{eqapplemsim}
H^i(G_{\kinfty}, W\otimes_C \calO C) \simeq 
H^i(\Gamma_\geom, \cm \otimes_{{\widehat{K^\pf_\infty}}} {\widehat{K^\pf_\infty}}[V_1, \cdots, V_d]), \quad \forall i \geq 0.
\end{equation*}
We claim the $\Gamma$-equivariant hence $\Gamma_\geom$-equivariant inclusion $\calH \into \cm$ induces an isomorphism
\begin{equation}\label{eqmandh}
H^i(\Gamma_\geom, \calH \otimes_{\hatkinfty} \hatkinfty[V_1, \cdots, V_d]) 
\simeq H^i(\Gamma_\geom, \cm \otimes_{{\widehat{K^\pf_\infty}}} {\widehat{K^\pf_\infty}}[V_1, \cdots, V_d]), 
\quad \forall i\geq 0. 
\end{equation}
The claim follows from similar argument in the paragraph below \cite[Lem. 6.17]{Sch13}. Indeed, the increasing $\mathbb{Z}^{\geq 0}$- filtration on ${\widehat{K^\pf_\infty}}[V_1, \cdots, V_d]$ (resp. on $\hatkinfty[V_1, \cdots, V_d]$) defined by polynomial degree is $\Gamma_\geom$-stable, and hence to prove \eqref{eqmandh}, it suffices to consider the graded pieces with respect to above filtrations. In both gradeds, the action of $\Gamma_\geom$ on  the $V_i$'s is trivial, and hence it reduces to show that 
\begin{equation}\label{eqhmpiece}
H^i(\Gamma_\geom, \calH) \simeq H^i(\Gamma_\geom, \cm), \quad \forall i \geq 0.
\end{equation}
Note \eqref{eq215lz} induces a $\Gamma_\geom$-equivariant  decomposition
\begin{equation*}
\cm =\calH \oplus (\bigoplus_{\tau\neq \mathbf 1} \cm_\tau)
\end{equation*}
Note for each $\tau\neq \mathbf 1$, there exists some $i$ such that $\gamma_i-1$ acts on $ \cm_\tau$ bijectively (by considering the eigenvalues); here $\gamma_i$ is a topological generator of a copy of $\zp(1)$ in $\Gamma_\geom \simeq \zp(1)^d$. This gives \eqref{eqhmpiece} and hence \eqref{eqmandh}.
Finally, using the projection map (modulo the $V_i$'s)
$$\calH \otimes_{\hatkinfty} \hatkinfty[V_1, \cdots, V_d] \to \calH,$$
and inductively applying \cite[Lem. 2.10]{LZ17}, we have
\begin{equation}\label{eqLZ210}
H^i(\Gamma_\geom, \calH \otimes_{\hatkinfty} \hatkinfty[V_1, \cdots, V_d]) \simeq \begin{cases} \calH  & i=0 \\ 0 & i>0.\end{cases}
\end{equation} 
This in particular shows there is a $\Gamma$-equivariant isomorphism
\begin{equation}\label{eqcompahavi}
\calH(W)\otimes_\hatkinfty \hatkinfty[V_1, \cdots, V_d] \simeq \calH \otimes_\hatkinfty \hatkinfty[V_1, \cdots, V_d].
\end{equation}
Via the surjection $G_K \onto \Gamma$  and using Eqn. \eqref{eqcalhW}, we have   $G_K$-equivariant isomorphisms
$$\calH(W)\otimes_\hatkinfty \coc \simeq \calH \otimes_\hatkinfty \coc \simeq   W\otimes_C \coc.$$

\end{proof}

\begin{proof}[\textbf{Proof of Thm. \ref{thmSimpson}(2-4)}]
We now prove remaining parts of  Thm. \ref{thmSimpson}.

Item (2): by Thm. \ref{thm514}, there is a $G_K$-equivariant isomorphism 
$$W\otimes_C \coc \simeq \calH(W)\otimes_{\hatkinfty} \coc$$
Hence we can obtain a $G_L$-equivariant isomorphism 
$$W\otimes_C \coc \otimes_\coc \mathcal{O}C_L \simeq \calH(W)\otimes_{\hatkinfty} \coc \otimes_\coc \mathcal{O}C_L,$$
which, after re-organization becomes
$$ (W\otimes_C C_L) \otimes_{C_L} \mathcal{O}C_L \simeq  (\calH(W)\otimes_{\hatkinfty} \widehat{L_\infty}) \otimes_{\widehat{L_\infty}} \mathcal{O}C_L.$$
Taking $G_{L_\infty}$-invariant on both sides gives 
$$  \calH_L(W\otimes_C C_L) \simeq \calH(W)\otimes_{\hatkinfty} \widehat{L_\infty}.$$
The decompleted  base change isomorphism then follows.


Item (3): The Higgs field on $\calH(W)= (W\otimes_C \coc)^{G_\kinfty}$ is naturally induced from $\gr^0 \nabla$ on $\coc$ defined in Notation \ref{notagrzero}, since $\gr^0 \nabla$  commutes with $G_K$-actions.


Item (4): it suffices to prove that $\calH$ is compatible with tensor products and duality, the upgraded version and the $H$-version immediately follow.
Let $W_1, W_2 \in \rep_{G_K}(C)$, then it is easy to check that
$$(\ch(W_1)\otimes_\hatkinfty \ch(W_2))\otimes_\hatkinfty \coc \simeq (W_1\otimes_C W_2)\otimes_C \coc.$$
Take $G_{\kinfty}$-invariants on both sides gives
\begin{equation}\label{eqtensorcalh}
\ch(W_1)\otimes_\hatkinfty \ch(W_2) \simeq \ch(W_1\otimes_C W_2).
\end{equation}
We now consider duality.
Apply \eqref{eqtensorcalh} to $W_1=W, W_2=W^\vee$, then we can obtain a homomorphism of finite dimensional $\hatkinfty$-vector spaces
$$\ch(W^\vee) \to \ch(W)^\vee.$$
It is an isomorphism because its base change over (the field) $\widehat{K^\pf_\infty}$ is so, since the base change, via Item (2), recovers duality morphism in Sen theory for the field $\kpf$ (with perfect residue field).
\end{proof}

\subsection{Hodge--Tate rigidity}\label{subsecHT}

\begin{defn}
For $W \in \rep_{G_K}(C)$, define 
$$D_{\HT, K}(W): = (W\otimes_C \obht)^{G_K}.$$ 
\end{defn}


\begin{prop}\label{propDHTBC}
There is a base change isomorphism
$$D_{\HT, K}(W)\otimes_K L \simeq D_{\HT, L}(W\otimes_C C_L).$$
\end{prop}
\begin{proof}
It is clear $D_\HT(W) \simeq \oplus_{j \in \mathbb Z} (\ch(W(j)))^{\gammak}$, where $W(j)$ is the Tate twist. It suffices to note that for each $j$, $(\ch(W(j)))^{\gammak} \otimes_K L  \simeq (\ch_L(W(j) ))^{\Gamma_L}$ by Cor. \ref{corhiggsbasechange}.
\end{proof}

\begin{cor} \label{cor523}
For $W \in \rep_{G_K}(C)$, $D_{\HT, K}(W)$ is a $K$-vector space such that 
$$\dim_K D_{\HT, K}(W) \leq \dim_C W.$$
\end{cor}
\begin{proof}
Apply above proposition to the embedding $K \into \kpf$, then it suffices to treat the case $K=\kpf$, which is well-known, cf. \cite[\S 2.7]{Fon04}.
\end{proof}
 
 \begin{defn}
Say $W \in \rep_{G_K}(C)$ a Hodge--Tate $C$-representation (or simply, $W$ is Hodge--Tate) if 
$$\dim_K D_{\HT, K}(W) = \dim_C W.$$
 \end{defn}

\begin{theorem}
\begin{enumerate}
\item   Let $W \in \rep_{G_K}(C)$, then $W$ is Hodge--Tate if and only the $G_L$-representation $C_L\otimes_C (W|_{G_L})$ is so.
  
\item   Let $V \in \rep_{G_K}(\qp)$, then $V$ is Hodge--Tate if and only $V|_{G_L}$ is Hodge--Tate.
\end{enumerate}
\end{theorem} 
\begin{proof}
Apply Prop. \ref{propDHTBC}.
\end{proof} 
 
We also discuss some results related with ``Sen operators" (which  is not covered in Brinon's Sen theory \cite{Bri03}).

\begin{theorem} \label{thm526semisimple}

\begin{enumerate}
\item
Let $H \in \rep_{\Gamma_K}(\kinfty)$,                                                          
 then the $\Gamma_K$-action on $H$ is locally analytic. As a consequence, one can define a $\kinfty$-linear endomorphism $\varphi_H: H \to H$, called the Sen operator.

\item Let $W \in \rep_{G_K}(C)$, then $W$ is Hodge--Tate if and only if the Sen operator on $\mathrm{H}(W)$ (defined via Thm. \ref{thmSimpson}) is semi-simple with integer eigenvalues.
\end{enumerate}
\end{theorem}
 \begin{proof}
 For Item (1), the base change $H \otimes_\kinfty K^\pf_\infty$ is an object in $\rep_{\Gamma_\kpf}(K^\pf_\infty)                                                                 $, with the $\Gamma_\kpf$-action locally analytic by Sen theory (in the perfect residue field case). Note that $\Gamma_\kpf < \Gamma_K$ is an open embedding, hence the $\Gamma_K$ action on $H$ is also locally analytic. Item (2) then follows immediately (from the case for $\kpf$).
 \end{proof}


\section{Sen--Fontaine theory}\label{secSenFontaine}
In this section, we first review the Sen--Fontaine theory in the classical perfect residue field case.
We then carry out a Tate--Sen--Fontaine decompletion in the imperfect residue field case. Finally, we also summarize Andreatta--Brinon's ``Sen--Fontaine theory"  in the imperfect residue field case; it serves for comparison purposes (with our Riemann--Hilbert correspondence) only, and will not be  used in our paper.
\subsection{Sen--Fontaine theory in the perfect residue field case} \label{subsecSenFon}
In continuation of \S \ref{subsec41}, we recall Sen--Fontaine theory for the    CDVF  $\bbk$ with perfect residue field.

\begin{notation}
 Let 
 $$\bbldr=(\bbdr)^{G_{\bbkinfty}},\quad \bbldrplus =(\bbdrplus)^{G_{\bbkinfty}}.$$
 The filtration on $\bbdr$ induces a filtration on $\bbldr$. Note $\fil^0 \bbldr =\bbldrplus$.
 \end{notation}
 
 \begin{prop}\cite[Thm. 3.6]{Fon04}
 The embeddings $\bbkinfty[[t]] \into  \bbldrplus \into \bbdrplus$ induce tensor equivalences of categories $$ \rep_{\gammabbk}(\bbkinfty[[t]]) \to \rep_{\gammabbk}(\bbldrplus) \to \rep_{G_\bbk}(\bbdrplus), $$
  as well as 
$$ \rep_{\gammabbk}(\bbkinfty((t)) \to \rep_{\gammabbk}(\bbldr) \to \rep_{G_\bbk}(\bbdr). $$
 \end{prop}

\begin{notation}
Define the following    functors:
\begin{eqnarray*}
\crh_{\bbk }^+: &   \rep_{G_\bbk}(\bbdrplus) \simeq \rep_{\gammabbk}(\bbldrplus) \\
\crh_{\bbk }: &    \rep_{G_\bbk}(\bbdr) \simeq \rep_{\gammabbk}(\bbldr) \\
\rrh_{\bbk }^+ : &    \rep_{G_\bbk}(\bbdrplus)  \simeq \rep_{\gammabbk}(\bbkinfty[[t]])\\
\rrh_{\bbk } : &    \rep_{G_\bbk}(\bbdr)  \simeq \rep_{\gammabbk}(\bbkinfty((t)))
\end{eqnarray*}
\end{notation}

\begin{cor}
Let $W  \in \rep_{G_\bbk}(\bbdrplus)$, then we have canonical isomorphisms
\begin{eqnarray*}
\crh_{\bbk}^+(W)/t &\simeq \calH_{\bbk}(W/tW)\\
\rrh_{\bbk}^+(W)/t &\simeq  \mathrm{H}_{\bbk}(W/tW)
\end{eqnarray*}
\end{cor}

\begin{rem}
\begin{enumerate}
\item
Let $V \in \rep_{G_\bbk}(\qp)$ and let $W=V\otimes_\qp \bbdrplus \in \rep_{G_\bbk}(\bbdrplus)$.
 Then   $\rrh_{\bbk}^+(W)/t  \simeq  \mathrm{H}_{\bbk}(W/tW)$ is precisely the usually written  isomorphism ``$D_\dif^+(V)/t \simeq  D_\Sen(V)$" in the literature, cf. e.g., \cite[\S 5.3]{Ber02}.

\item  
We can summarize Sen theory and Sen--Fontaine theory in the perfect residue field case as follows.
\begin{equation*}
\begin{tikzcd}
                                      & \rep_{G_\bbk}(\bbC) \arrow[rr, "\simeq"]                     &  & \rep_{\Gamma_\bbk}(\hatkinfty) \arrow[rr, "\simeq"]               &  & \rep_{\Gamma_\bbk}( {\kinfty})                               \\
\rep_{G_\bbk}(\qp) \arrow[ru] \arrow[rd] &                                                        &  &                                                                    &  &                                                            \\
                                      & \rep_{G_\bbk}(\bbdrplus) \arrow[rr, "\simeq"] \arrow[uu, "\bmod t"'] &  & {\rep_{\Gamma_\bbk}(\bbldrplus)} \arrow[rr, "\simeq"] \arrow[uu, "\bmod t"'] &  & {\rep_{\Gamma_\bbk}( {\kinfty}[[t]])} \arrow[uu, "\bmod t"']
\end{tikzcd}
\end{equation*}
\end{enumerate}
\end{rem}

\subsection{Tate--Sen--Fontaine decompletion in the imperfect residue field case}
In continuation of \S \ref{subsec43}, we develop some Tate--Sen--Fontaine decompletion results in the imperfect residue field case.

 \begin{lemma} \label{lem621}
 Let $U \in \rep_{G_K}({\bdrplus})$.
 Suppose $-\infty \leq a \leq b \leq +\infty$.
 \begin{enumerate}
 \item $H^i(\Ht, U\otimes_{\bdrplus} \obdr^{[a, b]})=0$ for $i \geq 1$.
 \item $H^i(G_{\kinfty}, U\otimes_{\bdrplus}\obdr^{[a, b]})=0$ for $i \geq 1$.
 \end{enumerate}
 \end{lemma}
\begin{proof}
Consider Item (1), the case $a=b$ follows from Lem. \ref{lemSimpcoho}. The case where $b, a \neq \pm \infty$ then follows from induction on the length $b-a$. The case $a=-\infty$ follows since $\Ht$ acts on $t$ trivially.
The case   $  b=+\infty$  follows from \cite[Lem. 3.18]{Sch13} (note: it is important that the filtration on $\obdr$ is \emph{complete} in order to apply \emph{loc. cit.}, cf. Rem. \ref{remcompleteobdr}).

The argument for Item (2) is similar, where the  $a=b$  case follows from Thm. \ref{thm514}.
\end{proof}

\begin{defn}
Let $\oldr: =(\obdr)^{G_{\kinfty}}$ and induce a filtration from that on $\obdr$.
\end{defn}

\begin{rem}
\begin{enumerate}
\item There is no need to define $\oldr^+$, cf. Rem. \ref{remobdrplus}.
\item By \cite[Lem. 2.31]{Bri06}, there is an inclusion
$$ K\otimes_{K^{\nabla=0}} \bdr^{ G_{\kinfty}=1} \subset \oldr. $$
We do not know if this is an equality.
\end{enumerate}
\end{rem}


\begin{lemma} \label{lemgroldr}
For $i \in \mathbb Z$, we have
$\gr^i  (\oldr) \simeq \hatkinfty(i)$ as $\Gamma_K$-representations.
\end{lemma} 
\begin{proof}
Consider the short exact sequence 
$$0 \to \obdr^{[-\infty, i]} \to \obdr^{[-\infty, i+1]} \to \gr^i \obdr \to 0.$$ By Lem. \ref{lem621}, it is still short exact after taking $G_{\kinfty}$-invariants. Now note $\gr^i \obdr =\coc(i)$, and by Thm. \ref{thmSimpson}, $(\coc(i))^{G_\kinfty} =\calH(C(i))=\hatkinfty(i).$
\end{proof}

\begin{prop}\label{thmTSFontaine}
Let $X \in \rep_{\Gamma_K}(\fil^0 \oldr)$, define
$$X_f = \projlim_{s \geq 1} (X/t^s X)_f, $$
where $(X/t^s X)_f$ denotes the union of finite dimensional $K$-sub-vector spaces of  $X/t^s X$ stable under $\gammak$-action. Then 
\begin{enumerate}
\item $X_f$ is also the union of $\kinfty[[t]]$-sub-modules of $X$ stable under $\gammak$-action; furthermore, 
$X_f$ is an object in $\rep_{\Gamma_K}(\kinfty[[t]])$.

\item In addition, the natural map
$$X_f\otimes_{\kinfty[[t]]} \fil^0 \oldr \to X$$
is an isomorphism.

\end{enumerate}
\end{prop}
\begin{proof}
If $K$ has perfect residue field, then this is \cite[Thm. 3.6]{Fon04}; the general case follows similar argument. Here, we briefly sketch the argument following the clear exposition of \cite[Prop. 2.11]{Shi18}. Indeed, it suffices to show:
\begin{itemize}
\item for each $s\geq 1$, the finite length representation $V=X/t^s X$ satisfies that $V_f \in \rep_{\gammak}(\kinfty[t]/t^s)$ and $V_f\otimes_{\kinfty[t]/t^s} \fil^0 \oldr/t^s \simeq V$. 
\end{itemize}
Note that when $s=1$, then this is precisely Prop. \ref{thmfon24}, because $ \fil^0 \oldr/(t) \simeq \hatkinfty$ by Lem. \ref{lemgroldr}. Suppose the above claim is true for $s\leq n-1$, and now consider $V=X/t^n X$. 
Let $ V'=t^{n-1}V$ and $V''=V/V'$, and consider the short exact sequence
$$0 \to V'  \to V \to V''  \to 0.$$
Pick any $\kinfty$-basis $\bar{v}_1, \cdots, \bar{v}_r$ of $V''_f$ (which is automatically a $\hatkinfty$-basis of $V''$ by induction hypothesis), and lift it to elements $v_1, \cdots, v_r \in V$, which by Nakayama Lemma, is a $ \fil^0 \oldr/t^n$-basis of $V$. We want to modify $v_i$ so that they fall into $V_f$. 
Indeed, given any $\gamma \in \gammak$, let $T$ be the matrix of its action with respect to the basis $v_1, \cdots, v_r$. As $\bar{v}_1, \cdots, \bar{v}_r$  is a basis of $V''_f$, the matrix $T$ is of the form
$$T =T^0 +t^{n-1}T^1, \quad \text{ where }  T^0 \in \Mat(\kinfty[t]/t^n),   \quad T^1 \in \Mat(\hatkinfty).$$
Let $U: =T^0 \pmod{t} \in \Mat(\kinfty)$ which has to be invertible since it is the matrix of $\gamma$ acting on $V/tV$. Choose a \emph{specific} $\gamma \neq 1$ close to $1$ so that the matrix satisfies $v_p(U-1) >c_3=\frac{p}{p-1}$. 
Then exactly the same argument of \cite[Claim 2.12]{Shi18} --the key ingredient being the normalized Tate traces, which we established in Rem. \ref{remts123}-- shows that after changing $v_1, \cdots, v_d$ via a matrix $1+t^{n-1}M$ with $M \in \Mat(\hatkinfty)$,  then the matrix of $\gamma$ acting on them falls inside $\Mat(\kinfty[t]/t^n)$, which quickly implies that these new basis elements --which we still denote as $v_i$--  are elements in $V_f$ (using the fact $\gamma^{\zp} \subset \gammak$ is of finite index). Finally, a standard argument shows these (new) $v_1, \cdots, v_d$ has to be a basis of $V_f$.
\end{proof}

\begin{cor}\label{corTSFontaine}
For $Y$ an object in $\rep_{\Gamma_K}(\kinfty[[t]])$, the base change $Y\otimes_{\kinfty[[t]]} \fil^0 \oldr$ is an object in $\rep_{\Gamma_K}(\fil^0 \oldr)$. This induces a tensor equivalence of categories: 
$$\rep_{\Gamma_K}(\kinfty[[t]]) \xrightarrow{\simeq } \rep_{\Gamma_K}(\fil^0 \oldr).$$
Similarly there is a  tensor equivalence of categories: 
$$\rep_{\Gamma_K}(\kinfty((t))) \xrightarrow{\simeq } \rep_{\Gamma_K}(\oldr).$$ 
\end{cor}

\subsection{Sen--Fontaine theory of Andreatta--Brinon}\label{subsecAB10}

By Prop. \ref{propucoor}, and note that $\Ht$ fixes the elements $[t_i^\flat]$, we have
$$ (\obdrplus)^{\Ht} = (\bdrplus)^{\Ht}[[u_1, \cdots, u_d]] = (\bdrplus)^{\Ht}[[tV_1, \cdots, tV_d]],$$
which then contains the ring $K^{(\pf)}_\infty[[t]][[tV_1, \cdots, tV_d]]$.

\begin{prop}  
There is a chain of tensor equivalences of categories
$$
\rep_\Gamma \left(K^{(\pf)}_\infty[[t, tV_1, \cdots, tV_d]] \right)
 \simeq  \rep_\Gamma\left( (\bdrplus)^{\Ht}[[tV_1, \cdots, tV_d]]  \right) \simeq  \rep_\gk\left(\obdrplus\right),$$
where all the functors are defined via base change.
\end{prop}
\begin{proof}
These are proved  in \cite[Cor. 3.8, Thm. 3.23]{AB10} in the (affine) \emph{relative} case; similar argument renders the imperfect residue field case, as reviewed in \cite{MoritadR}.
\end{proof}

\begin{remark}
 Let $M\in \rep_\Gamma (K^{(\pf)}_\infty[[t, tV_1, \cdots, tV_d]]) $, one can define Lie algebra actions and hence introduce various differential modules, cf. \cite[\S 4]{AB10} and \cite[\S 3.1]{MoritadR}. Indeed, Morita's proof of Thm. \ref{thmmorita} makes key use of these differential modules.
\end{remark}

\begin{rem} \label{remRelationABBrinon}
  The relation of Andreatta--Brinon's theory and Brinon's theory can be summarized in the following.
\begin{equation*}
\begin{tikzcd}
                                      & \rep_{G_K}(C) \arrow[r, "\simeq"]                                        & \rep_{\Gamma}({\widehat{K^\pf_\infty}}) \arrow[r, "\simeq"]                                           & \rep_{\Gamma}({{K^{(\pf)}_\infty}})                                                \\
\rep_{G_K}(\qp) \arrow[ru] \arrow[rd] &                                                                          &                                                                                                       &                                                                                    \\
                                      & {\rep_{G_K}(\obdrplus)} \arrow[uu, "\theta_K"'] \arrow[r, "\simeq"] & {\rep_{\Gamma}( (\obdrplus)^\Ht)} \arrow[uu, "\theta_K"'] \arrow[r, "\simeq"] & {\rep_{\Gamma}(K^{(\pf)}_\infty[[t, tV_1, \cdots, tV_d]])} \arrow[uu, "\theta_K"']
\end{tikzcd}
\end{equation*}
\end{rem}

 \section{$p$-adic Riemann--Hilbert correspondence and de Rham rigidity}
 \label{secRH}

 In this section, we construct the $p$-adic Riemann--Hilbert correspondence together with its decompleted version, and use it to prove a de Rham rigidity theorem.
 
\subsection{$p$-adic Riemann--Hilbert correspondence in the imperfect residue field case}
Recall in Construction \ref{consnablaobdr}, we have defined a connection operator   over the \emph{ring} $\obdr$. (In particular, our situation is simpler  than that in \cite[\S 3]{LZ17}, as we do not need to deal with sheaves.)

\begin{defn}
Let $\mathrm{MIC}_{\gammak}(\oldr)$ be the category where an object is some $M\in \rep_{\Gamma_K}(\oldr)$ equipped with:
\begin{itemize}
\item a $\Gamma_K$-stable $\mathbb Z$-filtration $\fil^i M$ such that $\fil^i M =t^i \fil^0 M, \forall i \in \mathbb Z$
\item a $\Gamma_K$-equivariant integrable connection $\nabla:  M \to M\otimes_K \wh{\Omega}_K$ satisfying Griffiths transversality, namely $\nabla (\fil^i M) \subset \fil^{i-1}M \otimes_K \wh{\Omega}_K, \forall i$.
\end{itemize}
\end{defn}

\begin{theorem}\label{thmRH}
For $ W\in \rep_{G_K}({\bdr})$, define 
$$  \crh(W):= (W\otimes_{\bdr} \obdr)^{G_{\kinfty}}$$
\begin{enumerate}
\item The above rule defines a functor 
$$\crh: \rep_{G_K}({\bdr}) \to  \rep_{\Gamma_K}(\oldr)$$
Via the equivalence in Cor. \ref{corTSFontaine}, we can define a functor 
$$\rrh:   \rep_{G_K}({\bdr}) \to  \rep_{\Gamma_K}(\kinfty((t))) $$

\item These functors satisfy base change properties with respect to the embedding $K \into L$, in the sense that we have canonical isomorphisms
\begin{align*}
\crh(W) \otimes_{ \mathcal{O}\mathbf{L}_\dR}  \mathcal{O}\mathbf{L}_{\dR, L} &\simeq \crh_{L}(W|_{G_L}\otimes_\bdr \mathbf{B}_{\dR, L}),\\
 \rrh(W) \otimes_{\kinfty((t))} L_\infty((t)) &\simeq \rrh_{L}(W|_{G_L}\otimes_{\bdr} \mathbf{B}_{\dR, L}),
\end{align*} 
where $W|_{G_L}\otimes_{\bdr} \mathbf{B}_{\dR, L}$ is regarded as an object in $\rep_{G_L}( \mathbf{B}_{\dR, L})$.

\item The functor $\crh$ can be upgraded to a functor
$$\rep_{G_K}({\bdr})\to \mathrm{MIC}_{\gammak}(\oldr).$$

\item The functors $\crh$ (as well as its upgraded version) and $\rrh$ are tensor functors and are compatible with duality.

\end{enumerate} 
\end{theorem} 

\begin{rem}
By abuse of terminology, we call both $\crh$ and its upgraded version the $p$-adic Riemann--Hilbert correspondence. (Indeed, the connection on $\crh(W)$ does not play any further role in this paper.) We call the functor $\rrh$  the decompleted $p$-adic Riemann--Hilbert correspondence.
 \end{rem}

\begin{proof}
Since $G_K$ is compact, one can easily show any $ W\in \rep_{G_K}({\bdr})$ admits a $G_K$-stable $\bdrplus$-lattice, cf.  the discussion above \cite[Thm. 3.9]{Fon04}. Hence for Item (1), it suffices to show for any $U \in \rep_{G_K}({\bdrplus})$ and for all $-\infty \leq a \leq b \leq +\infty$, the space 
$$
\crh^{[a, b]}(U):=H^0(G_{\kinfty}, U\otimes_{\bdrplus} \obdr^{[a, b]})
$$
 is a finite free module over  $\oldr^{[a, b]}:=(\obdr^{[a, b]})^{G_\kinfty}$, and the natural map
$$ \crh^{[a, b]}(U)\otimes_{\oldr^{[a, b]}}  \obdr^{[a, b]} \to U\otimes_{\bdrplus} \obdr^{[a, b]}$$
is a $G_K$-equivariant isomorphism.
The case when $a=b$ is precisely the (twisted) Simpson correspondence in Thm. \ref{thmSimpson}. The general case then follows  by standard d\'evissage argument, by making use of  the cohomology vanishing result of Lem. \ref{lem621}(2). Note that again, just as in Lem. \ref{lem621}, it is crucial we use the  \emph{filtration complete} ring $\obdr$.

 Items (2)-(4) are formal consequences of the above construction, and follows similar argument as in Thm. \ref{thmSimpson}.
\end{proof}

 \begin{cor}
 for $ U\in \rep_{G_K}(\bdrplus)$, denote 
$$  \crh^+(U):= \crh^{[0, +\infty]}(U)= (U\otimes_\bdrplus \fil^0 \obdr)^{G_{\kinfty}}.$$
It defines a tensor functor
$$\crh^+: \rep_{G_K}(\bdrplus) \to  \rep_{\Gamma_K}(\fil^0 \mathcal{O}\mathbf{L}_\dR).$$
Uia equivalence in Cor. \ref{corTSFontaine}, we have another tensor functor
$$\rrh^+: \rep_{G_K}(\bdrplus) \to  \rep_{\Gamma_K}(\kinfty[[t]]).$$
Then we have the following canonical isomorphisms:
\begin{align*}
\crh^+(U)/t & \simeq \calH(U\otimes_\bdrplus C), \\
\rrh^+(U)/t & \simeq  \mathrm{H}(U\otimes_\bdrplus C).
\end{align*} 
\end{cor}

\begin{remark}\label{remRelationHRH}
The relations of the Simpson and Riemann--Hilbert correspondences can be summarized in the following diagram.
 \begin{equation*}
\begin{tikzcd}
                                      & \rep_{G_K}(C) \arrow[r]                              & \rep_{\Gamma_K}(\hatkinfty) \arrow[r, "\simeq"]                                             & \rep_{\Gamma_K}(\kinfty)                              \\
\rep_{G_K}(\qp) \arrow[ru] \arrow[rd] &                                                      &                                                                                             &                                                       \\
                                      & \rep_{G_K}(\bdrplus) \arrow[r] \arrow[uu, "\bmod t"] & \rep_{\Gamma_K}(\fil^0 \mathcal{O}\mathbf{L}_\dR) \arrow[r, "\simeq"] \arrow[uu, "\bmod t"] & {\rep_{\Gamma_K}(\kinfty[[t]])} \arrow[uu, "\bmod t"]
\end{tikzcd}
 \end{equation*}
 
\end{remark}
 
 \subsection{de Rham rigidity}\label{subsecdR}
\begin{defn}
 For $W \in \rep_{G_K}(\bdr)$, define 
$$D_{\dR, K}(W): = (W\otimes_\bdr \obdr)^{G_K}.$$ 
\end{defn}

 \begin{prop} \label{prop721}
   Let $W \in \rep_{G_K}(\bdr)$.
 \begin{enumerate}
 \item For each $i\geq 0$, the cohomology space $H^i(\gammak, \crh(W))$ is a finite dimensional $K$-vector space, and vanishes when $i \geq 2$. 
 
 \item Furthermore, the natural map
$$ H^i(\gammak, \crh(W))\otimes_K {L} \to H^i(\Gamma_{L}, \crh_{L}(W|_{G_L}\otimes_{\bdr} \mathbf{B}_{\dR, L})) $$ is an isomorphism.
Hence in particular 
\begin{equation*}
D_{\dR, K}(W) \otimes_K {L} \simeq D_{\dR, {L}} (W|_{G_L}\otimes_{\bdr} \mathbf{B}_{\dR, L})
\end{equation*}
 \end{enumerate}
\end{prop}
\begin{proof}
Let $U  \in \rep_{G_K}(\bdrplus)$ be a $G_K$-stable lattice of $W$, and let $V=U/tU \in \rep_{G_K}(C)$. Let $V(j)$ denote the Tate twists.
 By standard d\'evissage, it suffices to prove for any $j \in \mathbb Z$ we have:
\begin{itemize}
\item  $H^i(\gammak, \ch(V(j))$ is a finite dimensional $K$-vector space, and vanishes when $i \geq 2$.

\item  Furthermore, the natural map
$$ H^i(\gammak, \ch(V(j)))\otimes_K {L} \to H^i(\Gamma_{L}, \ch_{L}(V(j)|_{G_{L}})\otimes_C C_L)) $$ is an isomorphism.
\end{itemize} 
These follow from Cor. \ref{corhiggsbasechange}.
\end{proof}

\begin{cor}   
For $W \in \rep_{G_K}(\bdr)$, $D_{\dR, K}(W)$ is a $K$-vector space such that 
$$\dim_K D_{\dR, K}(W) \leq \dim_\bdr W.$$
\end{cor}
\begin{proof}
Argue similarly as Cor. \ref{cor523}, using Prop. \ref{prop721}. 
\end{proof}

 \begin{defn}
Say $W \in \rep_{G_K}(\bdr)$ is de Rham if
$$\dim_K D_{\dR, K}(W) = \dim_\bdr W.$$ 
Say $U \in \rep_{G_K}(\bdrplus)$ is de Rham if $U[1/t]$ is so.
 \end{defn}

\begin{theorem}  \label{thmdrrig} 
\begin{enumerate}
\item Let $W \in \rep_{G_K}(\bdr)$. Then $W$ is de Rham if and only $W|_{G_L}\otimes_{\bdr} \mathbf{B}_{\dR, L}$ is de Rham.
\item  Let $V \in \rep_{G_K}(\qp)$, then $V$ is de Rham if and only $V|_{G_{L}}$ is de Rham.
\end{enumerate}
\end{theorem} 
\begin{proof}
Apply Prop. \ref{prop721}.
\end{proof}

 In the statement of the following theorem, we refer to the paragraph above \cite[Lem. III.1.3]{Ber08Ast} for a quick review of concepts such as  formal connections, regular connections, and trivial connections.
 
 \begin{theorem}\label{thmregularconndR}
 \begin{enumerate}
 \item The $\Gamma_K$-action on $\kinfty[[t]]$ is locally analytic, and the Lie algebra action induces a differential map $\nabla_\gamma: \kinfty[[t]] \to \kinfty[[t]]$ which is $\kinfty$-linear and $\nabla_\gamma(t)=t$.
Here $\nabla_\gamma$ is the Lie algebra operator defined by  $\nabla_\gamma:=\frac{\log(g)}{\log \chi_p(g)}$ for (any) $g \in \Gamma_K$ sufficiently close to $1$, where $\chi_p$ is the cyclotomic character.

 \item 
Let $Y \in  \rep_{\Gamma_K}(\kinfty[[t]])$, then the $\Gamma_K$-action on $Y$ is locally analytic. Thus the Lie algebra action induces a \emph{regular connection} $\nabla_\gamma: Y \to Y$.

 \item  Let $W \in \rep_\gk(\bdr)$. Then it is de Rham if and only the regular connection  $\nabla_\gamma: \rrh(W) \to \rrh(W)$   is \emph{trivial}.
 \end{enumerate}
 \end{theorem}
 \begin{proof}
 Item (1) is obvious. Item (2) and (3) follows similar argument as in Thm. \ref{thm526semisimple}. Indeed, for Item (2), the base change $Y\otimes_{\kinfty[[t]]} K^\pf_\infty[[t]]$ has locally analytic action by $\Gamma_\kpf$, by the classical theory \cite[Prop. 3.7]{Fon04}. 
 Hence $\gammak$ action on $Y$ is also locally analytic since $\Gamma_\kpf \simeq \Gamma_K$.

 Finally consider Item (3). First, note that for any $G_K$-stable lattice $U \in \rep_\gk(\bdrplus)$ of $W$, the connection  $\nabla_\gamma: \rrh^+(U) \to \rrh^+(U)$ is regular by Item (2), and hence the induced connection on $\rrh(W)$ is also regular.
 Now for brevity, let $X=W|_{G_\kpf}\otimes_\bdr \mathbf{B}_{\dR, \kpf}$.
  If $\nabla_\gamma$ on $\rrh(W)$ is trivial, then the similar connection  $\nabla_\gamma$ on $\rrh_\kpf(X)$ is also trivial, and hence $X$ is de Rham by the classical theory in the perfect residue field case \cite[Prop. 5.9]{Ber02} (which treated the case when $W=V\otimes_\qp \bdr$ for some $V=\rep_\gk(\qp)$, but the general case is the same).  Thus $W$ is de Rham by Thm. \ref{thmdrrig}. Conversely, if $W$ is de Rham, then  $D_{\dR, K}(W)$ is already a solution of  $\nabla_\gamma$ and hence the connection is trivial.
 \end{proof}

 \section{Appendix: a lemma on cyclotomic extension}
 
In this appendix, we record a proof of Lem. \ref{lem: intersect kinfty kpf} provided by   Brinon. This lemma is \emph{equivalent} to the isomorphism $\Gamma_{K^{(\pf)}} \simeq \gammak$
mentioned in  Notation \ref{notaGammagp}. These facts are used in the literature \cite{Bri03} etc., without a proof. 
In fact, for our main purpose (i.e., Sen theory), it would suffice to know that 
 $\Gamma_{K^{(\pf)}} < \gammak$ is an open embedding (which easily follows by looking at unifomizers on both sides of Lem. \ref{lem: intersect kinfty kpf}); indeed, the \emph{locally} analytic vectors we consider are insensitive to taking any \emph{open subgroups} of the relevant $p$-adic Lie groups. 
 Nonetheless, it is desirable to know the \emph{precise} Lem. \ref{lem: intersect kinfty kpf}, as  it should be useful in development of \emph{integral} $p$-adic Hodge theory: indeed, compare \cite[Prop. 4.1.5]{Liu10}   for some   similar looking (but very different) results (which are important for integral $p$-adic Hodge theory in \emph{loc. cit.}).

\begin{lemma} \label{lem: intersect kinfty kpf} (For $p$   any prime.)  
$\kinfty \cap K^{(\pf)}=K$.
\end{lemma}

\begin{remark} \label{rem: kinfty not tot ram}
We make some remarks and cautions before we prove   Lemma \ref{lem: intersect kinfty kpf}.
\begin{enumerate}
\item For any $K \subset_{\mathrm{fin}} E \subset K^{(\pf)}$, the extension $E/K$ is called a \emph{fiercely ramified extension} in the literature, cf. e.g. \cite{Wewers-fierce}. One should not say $E/K$ is   ``unramified" because the residual extension is not separable. 
Nonetheless, a uniformizer of $K$ is still a uniformizer of $E$.

\item It is a well-known fact that $\kinfty/K$ might not be totally ramified, cf. next item. Thus, in particular, one could not use elementary ``ramification" or ``uniformizer" argument (in conjunction with previous item) to conclude Lem. \ref{lem: intersect kinfty kpf}.

\item We thank Laurent Berger for providing the following example. 
 Let $p>2$. Let $x = (-1)^{(p-1)/2}p$, and let $y\in \mathbb{Z}_p^\times$ whose reduction modulo $p$ is not a square.
Let  $K = \qp(\sqrt{xy}).$ Note $K/\qp$ is totally ramified.
We claim the extension $K(\zeta_p)/K$ is   \emph{not} totally ramified. Suppose otherwise, then $K(\zeta_p)/\qp$ is also   totally ramified.
Note $\sqrt{x} \in \qp(\zeta_p) \subset  K(\zeta_p)$, thus $\sqrt{y}  \in K(\zeta_p)$. This contradicts with the fact that  $\qp(\sqrt{y})/\qp$ is   unramified.
\end{enumerate}
\end{remark}

The following proof is   due to Olivier Brinon. The main idea is to mix some ``ramification theory technique" (e.g., the trace in Step  2), together with some ``degree argument".

\begin{proof}[Proof of Lem. \ref{rem: kinfty not tot ram}]
We only prove the case when $[k_K:k_K^p]=p$, cf. Notation \ref{notatitflat}; the general case follows by standard induction argument (similar to Step 0 in the following). For notation simplicity, we write $t=t_1 \in K$  a lift of $\bar t$ which is a $p$-basis of $k_K$. So we shall   prove  that for any $m, n \geq 1$, 
\begin{equation} \label{eqmn}
K(\varepsilon_{m}) \cap K[t^{1/p^n}] =K
\end{equation} 
(Recall as in Notation \ref{nota: field K}, we use $\varepsilon_{m}$ to denote a $p^m$-th primitive root of unity).

\textbf{Step 0.} 
One observes that it suffices to treat the $n=1$ case. Indeed, then we can change $K$ to $K[t^{1/p^{n-1}}]$ and apply a standard induction argument to conclude the general case.

\textbf{Step 1.} 
Let $K'=K(\varepsilon_1)$ resp $K'=K(i)$ when $p>2$ resp $p=2$; here $i$ is a primitive $4$-th root of unity. 
We first prove 
\begin{equation} \label{kprimecapkt}
K' \cap K[t^{1/p}] =K
\end{equation} 
When $p>2$, this  holds because the degree of $K(\varepsilon_{1})/K$ is bounded by $p-1$ hence coprime to $p$.
Consider when $p=2$. If \eqref{kprimecapkt} does not hold, then because we are dealing with two degree 2 extensions of $K$, we must have
\[ K(i)= K[t^{1/2}]  \] 
 Then
$i=a+bt^{1/2}$ with $a, b \in K$. Take square on both sides and re-organize, one quickly concludes $a=0$, and thus $-1=b^2 t$. The elements $t$ and hence $b$ are units in $\mathcal{O}_K$; modulo the maximal ideal of  $\mathcal{O}_K$, one concludes $(\bar t)^{1/2} =(\bar t)^{1/p}\in k_K$, which is a contradiction.

\textbf{Step 2.} 
   We now prove,   
\begin{equation} \label{eqmnprime}
K'(\varepsilon_{m}) \cap K'[t^{1/p}] =K'.
\end{equation} 
Granting \eqref{eqmnprime}, we can  intersect both sides with $K[t^{1/p}]$, and apply \eqref{kprimecapkt} to  conclude \eqref{eqmn} for $n=1$ and hence proof of this lemma.

 Suppose $m \geq 2$ is the smallest number such that \eqref{eqmnprime} does not hold. 
 Note the Galois group $K'(\varepsilon_{m}) /K'$ is a cyclic group of $p$-power order; thus $K'(\varepsilon_{m}) \cap K'[t^{1/p}] $---being an intermediate extension--- can only be of form $K'(\varepsilon_{a})$ for some $a \leq m$. (This argument is the main reason we use $K'$ here). Thus, by minimality of $m$, we must have $a=m$; that is: 
 \[ K'(\varepsilon_{m}) \cap K'[t^{1/p}] =K'(\varepsilon_{m}) \neq K'\]
 Thus
 \[ K' \subset K'(\varepsilon_{m}) \subset K'[t^{1/p}] \]
 Eqn. \eqref{kprimecapkt} implies that  $K'[t^{1/p}]/K'$ is of degree $p$, thus  we must have
 \[ K'(\varepsilon_{m}) = K'[t^{1/p}]\]
Thus $[K'(\varepsilon_{m}): K']=p$ and hence $\varepsilon_{m-1} \in K'$. 

 Write
\begin{equation}\label{eq: write zeta}
\varepsilon_{m}=\sum_{i=0}^{p-1}\alpha_it^{i/p}
\end{equation} 
 with $\alpha_0,\ldots,\alpha_{p-1}\in K'$.  Let   $\text{Tr}:K'[t^{1/p}]\to K'$ be the trace map (of this degree $p$ cyclic extension).  
If $1\leq i \leq p-1$, the conjugates of $t^{i/p}$ over $K'$ are $\varepsilon_{1}^jt^{i/p}$ with $j\in\{0,\ldots,p-1\}$, so that $\text{Tr}(t^{i/p})=0$. Similarly, $\text{Tr}(\varepsilon_m)=0$ (note $\varepsilon_m \notin K'$!). Taking the trace on both sides of \eqref{eq: write zeta}, we see $\alpha_0=0$.

 More generally, if $\varepsilon_{m}t^{-i/p}\notin K'$, its conjugates over $K'$ are obtained by multiplying by powers of $\varepsilon_{1}$, so $\text{Tr}(\varepsilon_{m}t^{-i/p})=0$, which in turn implies that $\alpha_i=0$. This cannot hold for all $i\in\{0,\ldots,p-1\}$; thus there exists some $i$ such that
 \[\varepsilon_{m}t^{-i/p}\in K'. \]
  Note that both $\varepsilon_{m}$ and $t^{i/p}$ are invertible elements in the ring of integers of $K'[t^{1/p}]$: reducing modulo the maximal ideal, we deduce that $\overline{t}^{-i/p}$ whence $\overline{t}^{1/p}$ belong to the residue field of $K'$.
  This is impossible: as we pointed out earlier,  $K'[t^{1/p}]/K'$ is of degree $p$ because of \eqref{kprimecapkt}.
\end{proof}

 \bibliographystyle{alpha}

\begin{thebibliography}{DLLZ23}

\bibitem[AB10]{AB10}
Fabrizio Andreatta and Olivier Brinon.
\newblock {$B_{dR}$}-repr\'{e}sentations dans le cas relatif.
\newblock {\em Ann. Sci. \'{E}c. Norm. Sup\'{e}r. (4)}, 43(2):279--339, 2010.

\bibitem[BC08]{BC08}
Laurent Berger and Pierre Colmez.
\newblock {Familles de repr{\'e}sentations de de {R}ham et monodromie
  {$p$}-adique}.
\newblock {\em Ast{\'e}risque}, (319):303--337, 2008.
\newblock Repr{\'e}sentations $p$-adiques de groupes $p$-adiques. I.
  Repr{\'e}sentations galoisiennes et $(\phi,\Gamma)$-modules.

\bibitem[BC16]{BC16}
Laurent Berger and Pierre Colmez.
\newblock {Th{\'e}orie de {S}en et vecteurs localement analytiques}.
\newblock {\em Ann. Sci. {\'E}c. Norm. Sup{\'e}r. (4)}, 49(4):947--970, 2016.

\bibitem[Ber02]{Ber02}
Laurent Berger.
\newblock {Repr{\'e}sentations {$p$}-adiques et {\'e}quations
  diff{\'e}rentielles}.
\newblock {\em Invent. Math.}, 148(2):219--284, 2002.

\bibitem[Ber08]{Ber08Ast}
Laurent Berger.
\newblock {{\'E}quations diff{\'e}rentielles {$p$}-adiques et
  {$(\phi,N)$}-modules filtr{\'e}s}.
\newblock {\em Ast{\'e}risque}, (319):13--38, 2008.
\newblock Repr{\'e}sentations $p$-adiques de groupes $p$-adiques. I.
  Repr{\'e}sentations galoisiennes et $(\phi,\Gamma)$-modules.

\bibitem[Ber16]{Ber16}
Laurent Berger.
\newblock {Multivariable {$(\varphi,\Gamma)$}-modules and locally analytic
  vectors}.
\newblock {\em Duke Math. J.}, 165(18):3567--3595, 2016.

\bibitem[Bri03]{Bri03}
O.~Brinon.
\newblock Une g\'{e}n\'{e}ralisation de la th\'{e}orie de {S}en.
\newblock {\em Math. Ann.}, 327(4):793--813, 2003.

\bibitem[Bri06]{Bri06}
Olivier Brinon.
\newblock Repr\'{e}sentations cristallines dans le cas d'un corps r\'{e}siduel
  imparfait.
\newblock {\em Ann. Inst. Fourier (Grenoble)}, 56(4):919--999, 2006.

\bibitem[Bri08]{Bri08}
Olivier Brinon.
\newblock Repr\'{e}sentations {$p$}-adiques cristallines et de de {R}ham dans
  le cas relatif.
\newblock {\em M\'{e}m. Soc. Math. Fr. (N.S.)}, (112):vi+159, 2008.

\bibitem[BT08]{BT08}
Olivier Brinon and Fabien Trihan.
\newblock Repr\'{e}sentations cristallines et {$F$}-cristaux: le cas d'un corps
  r\'{e}siduel imparfait.
\newblock {\em Rend. Semin. Mat. Univ. Padova}, 119:141--171, 2008.

\bibitem[DLLZ23]{DLLZ}
Hansheng Diao, Kai-Wen Lan, Ruochuan Liu, and Xinwen Zhu.
\newblock Logarithmic {R}iemann-{H}ilbert correspondences for rigid varieties.
\newblock {\em J. Amer. Math. Soc.}, 36(2):483--562, 2023.

\bibitem[Fon04]{Fon04}
Jean-Marc Fontaine.
\newblock Arithm\'{e}tique des repr\'{e}sentations galoisiennes {$p$}-adiques.
\newblock Number 295, pages xi, 1--115. 2004.
\newblock Cohomologies $p$-adiques et applications arithm\'{e}tiques. III.

\bibitem[Gao]{Gaoimperf}
Hui Gao.
\newblock {Integral $p$-adic {H}odge theory in the imperfect residue field
  case}.
\newblock {\em preprint}.

\bibitem[HS97]{Hilton_Stammbach_homological_GTM_v2}
P.~J. Hilton and U.~Stammbach.
\newblock {\em A course in homological algebra}, volume~4 of {\em Graduate
  Texts in Mathematics}.
\newblock Springer-Verlag, New York, second edition, 1997.

\bibitem[Ked04]{Ked04}
Kiran~S. Kedlaya.
\newblock A {$p$}-adic local monodromy theorem.
\newblock {\em Ann. of Math. (2)}, 160(1):93--184, 2004.

\bibitem[Ked05]{Ked05}
Kiran~S. Kedlaya.
\newblock Slope filtrations revisited.
\newblock {\em Doc. Math.}, 10:447--525, 2005.

\bibitem[Liu10]{Liu10}
Tong Liu.
\newblock {A note on lattices in semi-stable representations}.
\newblock {\em Math. Ann.}, 346(1):117--138, 2010.

\bibitem[LZ17]{LZ17}
Ruochuan Liu and Xinwen Zhu.
\newblock Rigidity and a {R}iemann-{H}ilbert correspondence for {$p$}-adic
  local systems.
\newblock {\em Invent. Math.}, 207(1):291--343, 2017.

\bibitem[Mor10]{MoritadR}
Kazuma Morita.
\newblock Hodge-{T}ate and de {R}ham representations in the imperfect residue
  field case.
\newblock {\em Ann. Sci. \'{E}c. Norm. Sup\'{e}r. (4)}, 43(2):341--356, 2010.

\bibitem[Mor14]{Moritacrys}
Kazuma Morita.
\newblock Crystalline and semi-stable representations in the imperfect residue
  field case.
\newblock {\em Asian J. Math.}, 18(1):143--157, 2014.

\bibitem[Ohk11]{Ohk11}
Shun Ohkubo.
\newblock A note on {S}en's theory in the imperfect residue field case.
\newblock {\em Math. Z.}, 269(1-2):261--280, 2011.

\bibitem[Ohk13]{Ohk13}
Shun Ohkubo.
\newblock The {$p$}-adic monodromy theorem in the imperfect residue field case.
\newblock {\em Algebra Number Theory}, 7(8):1977--2037, 2013.

\bibitem[Pet23]{Pet23}
Alexander Petrov.
\newblock Geometrically irreducible {$p$}-adic local systems are de {R}ham up
  to a twist.
\newblock {\em Duke Math. J.}, 172(5):963--994, 2023.

\bibitem[Sch12]{Sch12}
Peter Scholze.
\newblock Perfectoid spaces.
\newblock {\em Publ. Math. Inst. Hautes \'{E}tudes Sci.}, 116:245--313, 2012.

\bibitem[Sch13]{Sch13}
Peter Scholze.
\newblock {$p$}-adic {H}odge theory for rigid-analytic varieties.
\newblock {\em Forum Math. Pi}, 1:e1, 77, 2013.

\bibitem[Sen81]{Sen81}
Shankar Sen.
\newblock Continuous cohomology and {$p$}-adic {G}alois representations.
\newblock {\em Invent. Math.}, 62(1):89--116, 1980/81.

\bibitem[Shi18]{Shi18}
Koji Shimizu.
\newblock Constancy of generalized {H}odge-{T}ate weights of a local system.
\newblock {\em Compos. Math.}, 154(12):2606--2642, 2018.

\bibitem[Shi22]{Shipst}
Koji Shimizu.
\newblock A {$p$}-adic monodromy theorem for de {R}ham local systems.
\newblock {\em Compos. Math.}, 158(12):2157--2205, 2022.

\bibitem[Tat67]{Tat67}
J.~T. Tate.
\newblock {$p$}-divisible groups.
\newblock In {\em Proc. {C}onf. {L}ocal {F}ields ({D}riebergen, 1966)}, pages
  158--183. Springer, Berlin, 1967.

\bibitem[Wew14]{Wewers-fierce}
Stefan Wewers.
\newblock Fiercely ramified cyclic extensions of {$p$}-adic fields with
  imperfect residue field.
\newblock {\em Manuscripta Math.}, 143(3-4):445--472, 2014.

\end{thebibliography}

\end{document}